\documentclass[11pt]{article}

\usepackage[utf8]{inputenc}
\usepackage[english]{babel}
\usepackage{hyperref}
\usepackage{amsmath,amsthm,enumerate}
\usepackage{enumitem} 
\usepackage{amssymb}
\usepackage{graphicx, color}
\usepackage{epsfig}
\usepackage{tikz}
\usetikzlibrary{calc,decorations.pathmorphing,decorations.text}
\usepackage{subcaption}
\usepackage{refcount}
\usepackage[hmargin=2.2cm,vmargin=1.8cm]{geometry}
\usepackage{mathtools, braket}
\usepackage{authblk}
\usepackage{centernot}

\usepackage{color}

\theoremstyle{plain}
\newtheorem{thm}{Thm}[section]

\newtheorem{theorem}[thm]{Theorem}
\newtheorem{lemma}[thm]{Lemma}
\newtheorem{corollary}[thm]{Corollary}
\newtheorem{proposition}[thm]{Proposition}
\newtheorem{conjecture}[thm]{Conjecture}
\newtheorem{problem}[thm]{Problem}

\newtheorem{observation}[thm]{Observation}
\newtheorem{definition}[thm]{Definition}

\newtheorem{innercustomgeneric}{\customgenericname}
\providecommand{\customgenericname}{}
\newcommand{\newcustomtheorem}[2]{%
	\newenvironment{#1}[1]
	{%
		\renewcommand\customgenericname{#2}%
		\renewcommand\theinnercustomgeneric{##1}%
		\innercustomgeneric
	}
	{\endinnercustomgeneric}
}

\newcustomtheorem{customtheorem}{Theorem}
\newcustomtheorem{customlemma}{Lemma}

\newenvironment{proof*}{\noindent\emph{Proof of the claim:}}{\hfill$\Diamond$}

\newcommand{\CK}{\!\scriptscriptstyle -k}

\renewcommand{\pod}[1]{\allowbreak\mathchoice
	{\if@display \mkern 0mu\else \mkern 0mu\fi (#1)}
	{\if@display \mkern 0mu\else \mkern 0mu\fi (#1)}
	{\mkern 1mu(\mathrm{mod}\mkern 4mu #1)}
	{\mkern 0mu(#1)}
}

\usepackage{caption}
\usetikzlibrary{matrix}
\usetikzlibrary{decorations.markings}
\tikzstyle{vertex}=[circle, draw, fill=black!50,
inner sep=0pt, minimum width=4pt]
\tikzset{->-/.style={decoration={
			markings,
			mark=at position .5 with {\arrow{>}}},postaction={decorate}}}

\tikzstyle{bigblue}=[color=blue, very thick, >=stealth]
\tikzstyle{lightblue}=[color=blue, thin, >=stealth]

\tikzstyle{bigred}=[color=red, very thick, >=stealth]
\tikzstyle{lightred}=[color=red, thin, >=stealth]

\tikzstyle{biggreen}=[color=black!30!green, very thick, >=stealth]
\tikzstyle{lightgreen}=[color=black!30!green,  thin, >=stealth]

\begin{document}
	
	\title{Circular flows in mono-directed signed graphs}
	
	\author[1]{Jiaao Li}
	\author[2]{Reza Naserasr}
	\author[2]{Zhouningxin Wang}
	\author[3]{Xuding Zhu}
	
	\affil[1]{\small School of Mathematical Sciences and LPMC, Nankai University, Tianjin 300071, China} 
	\affil[2]{\small Université Paris Cité, CNRS, IRIF, F-75013, Paris, France}
	\affil[3]{\small Departments of Mathematics, Zhejiang Normal University, Jinhua, China,  
	 	\linebreak 
		Emails: lijiaao@nankai.edu.cn; \{reza, wangzhou4\}@irif.fr;  xdzhu@zjnu.edu.cn }
	
	\maketitle
	
	\renewcommand{\baselinestretch}{1.2}
	\linespread{1.2}
	
	\begin{abstract}
		In this paper the concept of circular $r$-flows in a mono-directed signed graph $(G, \sigma)$ is introduced. That is a pair $(D, f)$, where $D$ is an orientation on $G$ and $f: E(G)\to (-r,r)$ satisfies that $|f(e)|\in [1, r-1]$ for each positive edge $e$ and $|f(e)|\in [0, \frac{r}{2}-1]\cup [\frac{r}{2}+1, r)$ for each negative edge $e$, and the total in-flow equals the total out-flow at each vertex. The circular flow index of a signed graph $(G, \sigma)$ with no positive bridge, denoted $\Phi_c(G,\sigma)$,  is the minimum $r$ such that $(G, \sigma)$ admits a circular $r$-flow. This is the dual notion of circular colorings and circular chromatic numbers of signed graphs recently introduced in [Circular chromatic number of signed graphs. R. Naserasr, Z. Wang, and X. Zhu. Electronic Journal of Combinatorics, 28(2)(2021), \#P2.44], and is distinct from the concept of circular flows in bi-directed graphs associated to signed graphs studied in the literature. We give several equivalent definitions, study basic properties of circular flows in mono-directed signed graphs, and explore relations with flows in graphs.
		Then we focus on upper bounds on $\Phi_c(G,\sigma)$ in terms of the edge-connectivity of $G$. Tutte's 5-flow conjecture is equivalent to saying that every 2-edge-connected signed graph has circular flow index at most $10$. We show that if $(G, \sigma)$ is a 3-edge-connected (4-edge-connected, respectively) signed graph,  then $\Phi_c(G,\sigma)\leq 6$ ($\Phi_c(G,\sigma)\leq 4$, respectively). For $k\geq 2$,  if $(G, \sigma)$ is $(3k-1)$-edge-connected, then $\Phi_c(G,\sigma)\leq \frac{2k}{k-1}$; if $(G, \sigma)$ is $3k$-edge-connected, then $\Phi_c(G,\sigma)< \frac{2k}{k-1}$; and if $(G, \sigma)$ is $(3k+1)$-edge-connected, then $\Phi_c(G,\sigma)\leq \frac{4k+2}{2k-1}$. When restricted to the class of signed Eulerian graphs, if $(G, \sigma)$ is $(6k-2)$-edge-connected, then $\Phi_c(G, \sigma)\leq \frac{4k}{2k-1}$. Applying this on planar graphs we conclude that every signed bipartite planar graph of negative-girth at least $6k-2$ admits a homomorphism to $C_{\!\scriptscriptstyle -2k}$. For the particular values of $r_{_k}=\frac{2k}{k-1}$, and when restricted to two natural subclasses of signed graphs, the existence of a circular $r_{_k} $-flow  is strongly connected with the existence of a modulo $k$-orientation, and in case of planar graphs, based on duality, with the homomorphisms to $C_{\!\scriptscriptstyle -k}$.

	\end{abstract}

	\section{Introduction}\label{sec:introduction}
	
	The concept of nowhere-zero integer flows in graphs, introduced by W.T. Tutte \cite{T54} in 1950s, is a central topic in graph theory. When restricted to planar graphs, this is the dual concept of   proper vertex coloring.    In 1988, A. Vince \cite{V88} introduced a   refinement of proper vertex coloring of graphs, which is now called the {\em circular coloring}. The dual notion, \emph{circular flow} in graphs, was introduced by L.A. Goddyn, M. Tarsi, and C.-Q. Zhang \cite{GTZ98} in 1998. 
	
	A \emph{signed graph}, denoted $(G, \sigma)$ (or $\hat{G}$  if the signature is clear from the context), is a graph $G$ together with a signature $\sigma:E(G)\to \{+,-\}$. In this work graphs are allowed to have multi-edges, but no loop. The notion of the circular coloring of graphs has been recently extended to signed graphs in \cite{NWZ21}. 
	As the dual of the circular coloring of graphs, we  define the circular flow in mono-directed signed graphs as follows. 
	
	\begin{definition}\label{def:r-flow}
		Given a signed graph $(G, \sigma)$ and a real number $r\ge 2$, a \emph{circular $r$-flow} in $(G, \sigma)$ is a pair $(D, f)$ where $D$ is an orientation on $G$ and $f: E(G)\to (-r, r)$ satisfies the following conditions:
		\begin{itemize}
			\item For each positive edge $e$ of $(G, \sigma)$, $|f(e)|\in [1, r-1]$.
			\item For each negative edge $e$ of $(G, \sigma)$, $|f(e)|\in [0, \frac{r}{2}-1]\cup[\frac{r}{2}+1, r)$.
			\item For each vertex $v$ of $(G, \sigma)$, the total out-flow equals the total in-flow under the orientation $D$, i.e.,

			$$\partial_{\scriptscriptstyle D} f(v):=\sum\limits_{e \in  \overleftarrow{E_{\scriptscriptstyle D}}(v) } f(e)-\sum\limits_{e \in  \overrightarrow{E_{\scriptscriptstyle D}}(v)} f(e)=0,$$
			
			where $\overleftarrow{E_{\scriptscriptstyle D}}(v)$ and $\overrightarrow{E_{\scriptscriptstyle D}}(v)$ denote the sets of out-arcs and in-arcs at $v$ respectively. 
		\end{itemize}
		The {\em circular flow index} of a signed graph $(G, \sigma)$, denoted $\Phi_c(G, \sigma)$,  is defined as 
		$$\Phi_c(G, \sigma) = \inf\{r: (G, \sigma) \text{ admits a circular $r$-flow} \}.$$	
	\end{definition}
As in the case of graphs, given a flow $(D, f)$ in $(G, \sigma)$, and for an edge $e$ of $G$, if $D'$ is obtained from $D$ by reversing the direction on $e$ and $f'$ is obtained from $f$ by changing the value at $e$ to $-f(e)$, then the pair $(D', f')$ is also a flow in $(G, \sigma)$. Thus, whenever needed, we may assume that $f(e)\geq 0$ for each edge $e$. Such a flow is called a \emph{non-negative} circular $r$-flow.
	
	It is easily observed that the third condition of Definition~\ref{def:r-flow} extends to all edge-cuts: For each subset $X$ of vertices, 
	\begin{equation}\label{equ:Delta-f=0}
	\partial_{\scriptscriptstyle D} f(X):=\sum\limits_{e \in  \overleftarrow{E_{\scriptscriptstyle D}}(X) } f(e)-\sum\limits_{e \in  \overrightarrow{E_{\scriptscriptstyle D}}(X)} f(e)=0.	\end{equation}
	
	The interval of allowed values for $|f(e)|$ is a closed interval of a circle of circumference $r$ obtained from $[0, r]$ by identifying the two end points. It follows from the compactness of the circle that the infimum in the definition of the circular flow index can be replaced with the minimum, even for infinite graphs (see \cite{BZ98}). In finite signed graphs, which is the focus of this work, using the notion of tight cut, to be introduced in Section \ref{sec:CircularFlow}, we will see that in fact the circular flow index is always a rational number and that it can be computed for any given signed graph, though the problem of computing it is in the class of NP-hard problems.

	Denoting by $(G, +)$ ($(G, -)$, respectively) the signed graph where all the edges are positive (negative, respectively), it follows from the definition that $\Phi_c(G, +) = \Phi_c(G)$. Thus our study of the circular flow index of signed graphs includes that of graphs. Another notion of circular flows in signed graphs in which edges are bi-directed  is studied in the literature (see for example \cite{B83} and \cite{RZ11}).  That concept was motivated by coloring of graphs embedded in non-orientable surface. However, there is no dual concept for circular flows in bi-directed graphs. 
	The concept of circular flow of signed graphs studied in this paper is exactly the dual of circular coloring of signed graphs studied in \cite{NWZ21}.
	
	In  Definition \ref{def:r-flow}, if the condition that $``\!\!\!\!\!\sum\limits_{e \in  \overleftarrow{E_{\scriptscriptstyle D}}(v) } f(e)-\sum\limits_{e \in  \overrightarrow{E_{\scriptscriptstyle D}}(v)} f(e)=0$ for every vertex $v$'' is replaced by the condition that $``\!\!\!\sum\limits_{e\in C^F} f(e)-\sum\limits_{e\in C^B} f(e)=0$ for every cycle $C$'' (where $C^F$ and $C^B$ denote the set of forward and backward arcs of $C$ respectively), then $f$ is called a \emph{circular $r$-tension}  in $(G, \sigma)$. The {\em circular chromatic number} of $(G, \sigma)$, denoted $\chi_c(G, \sigma)$, is the infimum of those $r$ for which $(G, \sigma)$ admits a circular $r$-coloring. Equivalently, it is the infimum of those $r$ for which $(G, \sigma)$ admits a circular $r$-tension. 
	
	Given a signed plane graph $(G, \sigma)$, the \emph{dual signed graph} of $(G, \sigma)$ is the signed plane graph $(G^*, \sigma^*)$ defined as follows: $G^*$ is the dual graph of the underlying graph $G$ and $\sigma^*(e^*)=\sigma(e)$ for each edge $e^*\in E(G^*)$ where $e^*$ is the dual edge of $e$.  It is easy to see that when restricted to signed plane graphs, a circular $r$-flow in $(G, \sigma)$ defined in this paper is equivalent to a circular $r$-tension of $(G^*, \sigma^*)$, and hence 
	$\Phi_c(G, \sigma) = \chi_c(G^*, \sigma^*)$. 
	
	\medskip
	
	A central topic in the study of circular flow in graphs is the relation between the edge-connectivity and the circular flow index.
	It is easily observed that, given a graph $G$, if a function $f$ satisfies Equation (\ref{equ:Delta-f=0}), then for every bridge $e$ of $G$ we must have $f(e)=0$. Thus a signed graph with a positive bridge admits no circular flow. For a signed graph $(G, \sigma)$ with no positive bridge, it is observed in Section~\ref{sec:CircularFlow},  based on the $6$-flow theorem of Seymour~\cite{S81}, that $\Phi_c(G, \sigma)\leq 12$. Note that if $(D,f)$ is a circular $r$-flow in $(G,\sigma)$ and $r' \ge r$, then $(D, \frac{r'}{r}f)$ is a circular $r'$-flow in $(G, \sigma)$.   
	 One of the most influential problems in this area is Jaeger's circular flow conjecture \cite{J88}.  Originally it asserted that for any positive integer $k$, every $4k$-edge-connected graph admits a circular $\frac{2k+1}{k}$-flow, observing that the $(4k-1)$-edge-connected graph $K_{4k}$ does not admit such a flow. This exact assertion was disproved in 2018 \cite{HLWZ18} for each integer $k\ge 3$. However, in support of the conjecture, many upper bounds for the circular flow indices of highly edge-connected graphs are provided. The best general results until now are that of \cite{LTWZ13} and \cite{LWZ20}. In \cite{LTWZ13} it is proved that the circular flow index of a $6k$-edge-connected graph is at most $\frac{2k+1}{k}$. In \cite{LWZ20} it is proved that the circular flow index of a $(6k+2)$-edge-connected graph is strictly less than $\frac{2k+1}{k}$, and that the circular flow index of a $(6k-2)$-edge-connected graph is at most $\frac{4k}{2k-1}$. For general values of $k$, none of these upper bounds are known to be tight and the search for the best upper bound continues.

	The restriction of Jaeger's circular flow conjecture to planar graphs has attracted a considerable amount of attention and remains open. 
	This restricted version, using the duality,  is stated as follows: Every planar graph of girth at least $4k$ admits a circular $\frac{2k+1}{k}$-coloring.  In \cite{KZ00, Zh02}, supported by the folding lemma, the odd-girth version of this conjecture was introduced. It asserts that every planar graph of odd-girth at least $4k+1$ has circular chromatic number at most $2+\frac{1}{k}$. The case $k=1$ is the well-known Gr\"{o}tzsch theorem. For $k=2$, the best known result of odd-girth $11$ follows from the lower bound on the edge-density of $C_{\!\scriptscriptstyle \, 5}$-critical graphs given in \cite{DP17}. Similarly, that
	every planar graph of odd-girth at least $17$ admits a circular $\frac{7}{3}$-coloring follows from \cite{PS22}. Alternative proofs of these two results (for $k=2,3$) are given in \cite{CL20}. For general values of $k$, the best supporting results follow from the general results on flows that are mentioned above.

	In this paper we first present some basic properties of circular flow in mono-directed signed graphs, and explore relations between flows in graphs and flows in signed graphs. In Section~\ref{sec:ModuloOrientation} we consider the special case of circular $\frac{2k}{k-1}$-flow. After extending the notion of modulo $k$-orientation to signed graphs we show a strong connection between the existence of a modulo $k$-orientation and the existence of a circular $\frac{2k}{k-1}$-flow.

	The main focus of this work, presented in Sections~\ref{sec:Main} and \ref{sec:Application to planar graphs}, is mostly on the relation between the edge-connectivity and the circular flow index of signed graphs. We observe that Tutte's 5-flow conjecture is equivalent to the statement that for any 2-edge-connected graph $G$ and any signature $\sigma$, $\Phi_c(G,\sigma)\leq 10$, and  Seymour's $6$-flow theorem is equivalent to   $\Phi_c(G,\sigma)\leq 12$. We show that if $G$ is 3-edge-connected (4-edge-connected, respectively), then $\Phi_c(G,\sigma)\leq 6$ ($\Phi_c(G,\sigma)\leq 4$, respectively). For $k\geq 2$, if $G$ is $(3k-1)$-edge-connected, then $\Phi_c(G,\sigma)\leq \frac{2k}{k-1}$; if $G$ is $3k$-edge-connected, then $\Phi_c(G,\sigma)< \frac{2k}{k-1}$; and if $G$ is $(3k+1)$-edge-connected, then $\Phi_c(G,\sigma)\leq \frac{4k+2}{2k-1}$. When restricted to the class of signed Eulerian graphs, if $G$ is $(6k-2)$-edge-connected, then $\Phi_c(G, \sigma)\leq \frac{4k}{2k-1}$. Applying this on planar graphs we conclude that every signed bipartite planar graph of negative-girth at least $6k-2$ admits a homomorphism to $C_{\!\scriptscriptstyle -2k}$.

	\medskip
	 In this paper, we will basically follow the notation of \cite{W96}. A \emph{cycle} of a graph $G$ is a $2$-connected subgraph all whose vertices are of degree $2$ (in the subgraph). A subgraph where the degree of each vertex is even is an \emph{even-degree} subgraph, and a connected even-degree subgraph is an \emph{Eulerian} subgraph. Given a proper and nonempty subset $X$ of vertices, the edge-cut (or cut for short) consisting edges with exactly one end in $X$ is denoted by $(X, X^c)$. A graph $G$ or a signed graph $(G, \sigma)$ is \emph{$k$-edge-connected} if each edge-cut of $G$ has at least $k$ edges.	Given a graph $G$, we denote by $kG$ the graph obtained from $G$ by replacing each edge with $k$ parallel edges.

	Given an orientation $D$ on $G$, the \emph{out-degree} of a vertex $v$ is denoted by $\overleftarrow{d_{\scriptscriptstyle D}}(v)$ and the \emph{in-degree} by $\overrightarrow{d_{\scriptscriptstyle D}}(v)$. In a signed graph $\hat{G}$, the set of positive edges (respectively, negative edges) are denoted by $E_{\scriptscriptstyle \hat{G}}^+$ (respectively, $E_{\scriptscriptstyle \hat{G}}^-$). The \emph{positive degree} of a vertex $v$ is the number of positive edges incident to $v$ and is denoted by $d^{+}(v)$. The \emph{negative degree}, denoted $d^{-}(v)$, is defined analogously. Given an orientation $D$ on a signed graph $\hat{G}$, we have the following four notions: \emph{positive in-degree}, \emph{positive out-degree}, \emph{negative in-degree} and \emph{negative out-degree} which are denoted, respectively, 
	$\overrightarrow{d^+_{\scriptscriptstyle \hat{G}}}(v)$, 
	$\overleftarrow{d^+_{\scriptscriptstyle \hat{G}}}(v)$, 
	$\overrightarrow{d^-_{\scriptscriptstyle \hat{G}}}(v)$, 
	$\overleftarrow{d^-_{\scriptscriptstyle \hat{G}}}(v)$. 
	In each of these notations, we may drop  the subscript when it is clear from the context.
	
	Given a signed graph $(G, \sigma)$ and a vertex $v$, \emph{switching} at $v$ is to switch the signs of all the edges incident to $v$ in $(G, \sigma)$. We say two signed graphs are \emph{switching equivalent} if one can be obtained from the other by switching at some vertices.   
	Given a signed graph $\hat{G}=(G, \sigma)$, the signed graph obtained from $\hat{G}$ by changing the signs of all the edges is denoted by $-\hat{G}$, that is to say, $-\hat{G}=(G, -\sigma)$. A \emph{negative cycle} of length $k$, denoted $C_{\CK}$, is the signed graph on the cycle of length $k$ with only one negative edge. It is switching equivalent to any signed graph on $C_{\scriptscriptstyle k}$ with an odd number of negative edges. One may observe that for each even integer $k$, the two signed cycles $C_{\CK}$ and $-C_{\CK}$ are switching equivalent. But for each odd value $k$, the signed cycle $-C_{\CK}$ is switching equivalent to $C_{\scriptscriptstyle k}$.

	\section{Switching and inversing in signed graphs}

	The circular chromatic number of a signed graph, as well as many other parameters of signed graphs, are invariant under the switching operation. The circular flow index, however, is invariant under the dual of switching operation. Given a signed graph $\hat{G}$ and a cycle $C$ of $\hat{G}$, \emph{inversing} on $C$ is to change the signs of all the edges of $C$. Two signed graphs $(G, \sigma)$ and $(G, \sigma')$ are said to be \emph{inversing equivalent} if one can be obtained from the other by inversing on some cycles.
	
	The combination of a sequence of switching at vertices is equivalent to switching at an edge-cut, i.e., change the signs of all the edges in an edge-cut. Similarly, the combination of a sequence of inversing on cycles is to change the signs of all the edges in an even-degree subgraph. The \emph{sign} of a multiset $F$ of edges in $(G, \sigma)$ is $\prod\limits_{e \in F}\sigma(e)$ where multiplicity counts. Switching at vertices does not change the sign of a cycle, and inversing on cycles does not change the sign of a cut. 
	It was proved by Zaslavsky~\cite{Z82} that $(G, \sigma)$ and $(G, \sigma')$ are switching equivalent if and only if they have the same set of negative cycles. The dual statement follows from the fact that the symmetric difference of two inversing-equivalent signatures induces an even-degree subgraph:

	\begin{lemma}\label{lem:Cycle-Switch-Equiv}
		Two signed graphs $(G, \sigma)$ and $(G, \sigma')$ are inversing equivalent if and only if they have the same set of negative edge-cuts.
	\end{lemma}

	Note that for  $(G, \sigma)$ and $(G, \sigma')$ to have the same set of negative cuts, it suffices that they have the same set of negative cuts among a basis of the bond-space. In particular, it suffices to consider all but one of the cuts $(\{v\}, V(G)\setminus\{v\})$ from each connected component. 
	As the sign of the cuts in a basis of the bond-space are independent of each other, we have the following corollary.
	
	\begin{corollary}
		Given a graph $G$ on $n$ vertices with $c$ connected components, there are $2^{n-c}$ distinct inversing-equivalent classes of signed graphs on $G$.
	\end{corollary}

	\begin{lemma}\label{lem:SpanningTree}
		Let $G$ be a connected graph and let $T$ be a spanning tree of $G$. Given a signed graph $(G, \sigma)$, there exists an inversing-equivalent signed graph $(G, \sigma')$ of $(G, \sigma)$ such that all the negative edges are in $E(T)$.
	\end{lemma}
	
	\begin{proof}
		For each edge $e\not\in T$, let $C_e$ be the unique cycle in the graph obtained from adding $e$ to $T$. If $e$ is a negative edge in $(G, \sigma)$, then apply an inversing on $C_e$. After applying this process on all the edges of $E(G)\setminus E(T)$, we obtain a signed graph $(G, \sigma')$ in which each edge in $E(G)\setminus E(T)$ is positive, proving this lemma.
	\end{proof}

	Depending on the sign and the parity of the number of edges, we have four types of cuts in a signed graph: a \emph{type $00$} (respectively, \emph{type $01$}, \emph{type $10$}, \emph{type $11$}) cut is a positive cut with  an even number of edges (respectively,  a positive cut  with  an odd number of edges,   a negative cut  with  an even number of edges,   a negative cut with an odd number of edges). For a signed graph $(G, \sigma)$, let \emph{$c_{\scriptscriptstyle ij}(G, \sigma)$} be the size (i.e., the number of edges) of a smallest cut of $(G, \sigma)$ of type $ij$ where $ij \in \mathbb{Z}_2^2$. When there is no such a cut, we write $c_{\scriptscriptstyle ij}(G, \sigma) =\infty$ if $ij\neq 00$,  and $c_{\scriptscriptstyle 00}(G, \sigma) =0$ for technical reasons.
	
	Let $-C^*_{\CK}$ be the signed graph on $kK_2$ with an odd number of positive edges, that is the dual of a plane $-C_{\CK}$.  The signed graph $-C^*_{\CK}$ is a key example in the study of signed graphs equipped with the inversing operation. Observe that $c_{\scriptscriptstyle 00}(-C^*_{\CK})=0$, $c_{\scriptscriptstyle 11}(-C^*_{\scriptscriptstyle -k})=\infty$ and that of the two values of $c_{\scriptscriptstyle 01}(-C^*_{\CK})$ and $c_{\scriptscriptstyle 10}(-C^*_{\CK})$, depending on the parity of $k$, one is $k$ and the other is $\infty$.
	
	There are two classes of signed graphs that are of special importance in the study of circular flow.
	\begin{itemize}
		\item For odd value $k$, a signed graph $(G, \sigma)$ satisfying that $c_{\scriptscriptstyle ij}(G, \sigma) \geq c_{\scriptscriptstyle ij}(-C^*_{\CK})$, for any $ij \in \mathbb{Z}_2^2$, is inversing equivalent to $(G,+)$.  
		\item For even value $k$, a signed graph $(G, \sigma)$ satisfying that $c_{\scriptscriptstyle ij}(G, \sigma) \geq c_{\scriptscriptstyle ij}(-C^*_{\CK})$, for any $ij \in \mathbb{Z}_2^2$, has all of its vertices to be of even degrees.
	\end{itemize}
	The restriction on the first class captures the classic graphs, and the second  is the class of even-degree (Eulerian if connected) signed graphs.

	The study of signed graphs with inversing operations is implicit in the study of $T$-joins of graphs. Given a graph $G$ and a subset $T$ of $V(G)$ where $|T|$ is even, a \emph{$T$-join} of $G$ is a spanning subgraph of $G$ where $T$ is the set of the odd-degree vertices. Given a signed graph $(G, \sigma)$, let $T$ be the set of vertices $v$ where the edge-cut $(\{v\}, V(G)\setminus \{v\})$ is negative. It is then easily observed that a subset $E'$ of edges is the set of negative edges of a signature $\sigma'$, inversing equivalent to $\sigma$, if and only if $E'$ induces a $T$-join. For more on $T$-joins we refer to \cite{C01}.

	\section{Equivalent definitions and basic properties}\label{sec:CircularFlow}
	
	In this section, we give some equivalent definitions and basic properties of circular $r$-flow in signed graphs. Most of these properties are direct extensions from graphs, and we shall omit some proofs.
	
	For $r=\frac{p}{q}$, in defining a circular $r$-flow in a signed graph, by adjusting some small portion of flow values on certain cycles, it can be shown that we may restrict ourselves to the values of the form $\frac{i}{q}$, for $i\in \{0, 1, \ldots, p-1\}$. Then multiplying all values by $q$, we may work with integer values. 
	
	\begin{definition}\label{def:p/q-flow}
		Given an even integer $p$ and an integer $q$ where $q\leq \frac{p}{2}$, a \emph{$(p,q)$-flow} in a signed graph $(G, \sigma)$ is a pair $(D, f)$ where $D$ is an orientation on $G$ and $f: E(G)\to \mathbb{Z}$ satisfies the following conditions. 
		\begin{itemize}
			\item For each positive edge $e$ of $(G, \sigma)$, $f(e)\in \{q, \ldots, p-q\}$.
			\item For each negative edge $e$ of $(G, \sigma)$, $f(e)\in \{0, \ldots, \dfrac{p}{2}-q\} \cup \{\dfrac{p}{2}+q, \ldots, p-1\}$.
			\item For each vertex $v$ of $(G, \sigma)$, $\sum\limits_{e \in \overleftarrow{E_{\scriptscriptstyle D}}(v) } f(e)=\sum\limits_{e \in \overrightarrow{E_{\scriptscriptstyle D}}(v) } f(e)$ (i.e., $\partial_{\scriptscriptstyle D} f(v)=0$).
		\end{itemize}
	\end{definition}

	In the study of integer flow, Tutte \cite{T54} introduced the concept of modulo $k$-flow in a graph $G$ as a pair $(D,f)$, where $D$ is an orientation of $G$ and   $f: E(D) \to \mathbb{Z}$ satisfies  $\partial_D f(v) \equiv 0 ~\pmod{k}$, and proved the following lemma.  
	
	\begin{lemma}{\em \cite{T54}}\label{lem:Tutte}
		If a graph admits a modulo $k$-flow $(D, f)$, then it admits an integer $k$-flow $(D, f’)$ such that $f'(e)\equiv f(e)~\pmod k$ for every edge $e$.
	\end{lemma}

	\begin{definition}\label{def:p/q mod-flow}
		Given an even integer $p$ and an integer $q$ where $q\leq \frac{p}{2}$, a \emph{modulo $(p,q)$-flow} in $(G, \sigma)$ is a pair $(D, f)$ where $D$ is an orientation on $G$ and $f: E(G)\to \mathbb{Z}_p$ satisfies the following conditions.
		\begin{itemize}
			\item For each positive edge $e$ of $(G, \sigma)$, $f(e)\in \{q, \ldots, p-q\}$.
			\item For each negative edge $e$ of $(G, \sigma)$, $f(e)\in \{0, \ldots, \dfrac{p}{2}-q\}\cup \{ \dfrac{p}{2}+q,\ldots,  p-1\}$.
			\item For each vertex $v$ of $(G, \sigma)$, $\sum\limits_{e \in \overleftarrow{E_{\scriptscriptstyle D}}(v) } f(e) \equiv \sum\limits_{e \in \overrightarrow{E_{\scriptscriptstyle D}}(v) } f(e)  ~\pmod p$ (i.e., $\partial_{\scriptscriptstyle D} f(v)\equiv 0 ~\pmod p$).
		\end{itemize}
	\end{definition}
	
	It follows from Lemma \ref{lem:Tutte} that $(G, \sigma)$ admits a $(p,q)$-flow if and only if $(G, \sigma)$ admits a modulo $(p,q)$-flow.  
	
	The notion of $(p,q)$-flow is the discrete version of circular $\frac{p}{q}$-flow. Similarly, the modulo $(p,q)$-flow is the discrete version of circular modulo $\frac{p}{q}$-flow defined as follows: 
	
	\begin{definition}\label{def:p/q-modulo-flow}
		Given a signed graph $(G, \sigma)$ and a real number $r\ge 2$, a \emph{circular modulo $r$-flow} in $(G, \sigma)$ is a pair $(D, f)$ where $D$ is an orientation on $G$ and $f: E(G)\to [0, r)$ satisfies the following conditions:
		\begin{itemize}
			\item For each positive edge $e$ of $(G, \sigma)$, $f(e) \in [1, r-1]$.
			\item For each negative edge $e$ of $(G, \sigma)$, $f(e) \in [0, \dfrac{r}{2}-1]\cup [\dfrac{r}{2}+1, r)$.
			\item For each vertex $v$ of $(G, \sigma)$,   $\sum\limits_{e \in \overleftarrow{E_{\scriptscriptstyle D}}(v) } f(e) \equiv \sum\limits_{e \in \overrightarrow{E_{\scriptscriptstyle D}}(v) } f(e)  ~\pmod r$ (i.e., $\partial_{\scriptscriptstyle D} f(v)\equiv 0 ~\pmod r$).
		\end{itemize}
	\end{definition}
	
	It follows from the discussion above that:
	
	\begin{theorem}\label{thm:Equiavlences}
		Given a signed graph $(G, \sigma)$ and a rational number $r=\frac{p}{q}$ where $p$ is even and $p\geq 2q$, the following claims are equivalent:
		\begin{enumerate}[label=(\arabic*)]
			\item $(G, \sigma)$ admits a circular $r$-flow.
			\item $(G, \sigma)$ admits a $(p,q)$-flow.
			\item $(G, \sigma)$ admits a modulo $(p,q)$-flow.
			\item $(G, \sigma)$ admits a circular modulo $r$-flow.
		\end{enumerate} 
	\end{theorem} 
	
	\medskip
	
	Taking Definition~\ref{def:p/q-modulo-flow}, one observes that if $(D, f)$ is a circular modulo $r$-flow in $(G, \sigma)$ and $(G, \sigma')$ is obtained from $(G, \sigma)$ by inversing on a cycle $C$, then $(D, f')$ is a circular modulo $r$-flow in $(G, \sigma')$ where $f'(e)=f(e)-\frac{r}{2}$ for each edge $e\in E(C)$ and $f'(e)=f(e)$ for all other edges. Thus the circular flow index is invariant under inversing operation.

Let us consider a circular $r$-flow $(D,f)$, based on Definition~\ref{def:r-flow} where we take a non-negative flow,  in a given signed graph $\hat{G}$. Recall that for a positive edge $e$, it is required that  $1 \leq f(e) \leq r-1$, and for a negative edge $e$, either $0 \leq f(e) \leq \frac{r}{2}-1$ or $\frac{r}{2}+1 \leq f(e) \leq r$. Thus a signed graph $\hat{G}$ admits a circular $r$-flow if and only if there is an orientation $D$, a partition $\pi=(E^-_1, E^-_2)$ of $E^-_{\scriptscriptstyle \hat{G}}$, and a function $f: E\to \mathbb{R}_{\geq 0}$ satisfying that    $$\sum\limits_{e \in  \overleftarrow{E_{\scriptscriptstyle D}}(v) } f(e)-\sum\limits_{e \in  \overrightarrow{E_{\scriptscriptstyle D}}(v)} f(e)=0,$$ and $s(e) \le f(e) \le t(e)$ for every edge $e \in E(\hat{G})$,
	where
	
	\vspace{.3cm}
	\begin{minipage}{0.4\textwidth}
		$$
		s(e) = \begin{cases} 1, &\text{ if $e \in E^+_{\scriptscriptstyle \hat{G}}$,} \cr
		0 , &\text{ if $e \in E^-_1$,} \cr
		\dfrac{r}{2}+1 , &\text{ if $e \in E^-_2$,} \cr
		\end{cases}
		$$
	\end{minipage}
	\begin{minipage}{0.1\textwidth}
		\text{~~and}
	\end{minipage}
	\begin{minipage}{0.4\textwidth}
		$$
		t(e) = \begin{cases} r-1, &\text{ if $e \in E^+_{\scriptscriptstyle \hat{G}}$,} \cr
		\dfrac{r}{2}-1 , &\text{ if $e \in E^-_1$,} \cr
		r , &\text{ if $e \in E^-_2$.} \cr
		\end{cases}
		$$
	\end{minipage} 
	\vspace{.3cm}

	By applying  Hoffman's circulation theorem \cite{H60},  we have the following lemma, which can be also viewed as an equivalent definition of the existence of a circular $r$-flow.
	
	\begin{lemma}\label{lem:orientcut}
		A signed graph $\hat{G}$ admits a circular $r$-flow if and only if there is an orientation $D$ and a partition $\pi=(E^-_1, E^-_2)$ of $E^-_{\scriptscriptstyle \hat{G}}$ such that for every cut $(X, X^c)$,   $\sum\limits_{e \in \overleftarrow{(X,X^c)}}s(e) \le  \sum\limits_{e \in \overrightarrow{(X,X^c)}} t(e)$, 
		where the function $s$ and $t$ are as defined above, and $\overleftarrow{(X,X^c)}$ and $\overrightarrow{(X,X^c)}$ are the set of arcs oriented from $X^c$ to $X$  and the set of arcs oriented from $X$ to $X^c$, respectively.
	\end{lemma}

	Note that given an orientation $D$ on a signed graph $\hat{G}$ and a partition $\pi=(E^-_1, E^-_2)$ of $E^-_{\scriptscriptstyle \hat{G}}$, there is a polynomial time algorithm that checks whether $\sum\limits_{e \in \overleftarrow{(X,X^c)}}s(e) \le  \sum\limits_{e \in \overrightarrow{(X,X^c)}} t(e)$ for every cut $(X, X^c)$. 
	However, it is an NP-hard problem to determine the circular flow index of an input (signed) graph   \cite{EMT16}.

	\begin{definition}
		\label{tight-cut}
		Given a signed graph $(G, \sigma)$ and a non-negative circular $r$-flow $(D, f)$ of $(G, \sigma)$, a positive edge $e$ is said to be \emph{tight} if either $f(e)=1$ or $f(e)=r-1$, and a negative edge $e$ is said to be \emph{tight} if either $f(e)=\frac{r}{2}-1$ or $f(e)=\frac{r}{2}+1$. A cut $(X, X^c)$ of $(G, \sigma)$ is said to be \emph{tight} with respect to $(D, f)$ if 
		$$f(e)=\begin{cases} 
		1, &\text{ if $e$ is a positive edge in $\overleftarrow{(X,X^c)}$,} \cr
		r-1, &\text{ if $e$ is a positive edge in $\overrightarrow{(X,X^c)}$,} \cr
		\frac{r}{2}+1, &\text{ if $e$ is a negative edge in $\overleftarrow{(X,X^c)}$,} \cr
		\frac{r}{2}-1, &\text{  if $e$ is a negative edge in $\overrightarrow{(X,X^c)}$.}
		\end{cases}$$
	\end{definition}

	The concept of tight cut in the circular flow is the dual notion of tight cycle in the circular coloring. Similar to the case of circular coloring, this notion can be used to provide an equivalent definition of the circular flow index. This is of particular importance to compute the exact value of the circular flow index of a given signed graph.
	
	\begin{lemma}\label{lem:tightcut}
		Given a signed graph $(G, \sigma)$, $\Phi_c(G, \sigma)=r$ if and only if $(G, \sigma)$ admits a circular $r$-flow and for any non-negative circular $r$-flow $(D,f)$, $(G, \sigma)$ has a tight cut with respect to $(D,f)$.
	\end{lemma}
	
	\begin{proof}
		We first prove the ``only if'' part. Assume that $\Phi_c(G, \sigma)=r$ but $(G, \sigma)$ has no tight cut with respect to some non-negative circular $r$-flows. We choose such a circular $r$-flow $(D, f)$ with minimum number of tight edges. First we claim that there is no tight edge in $(G, \sigma)$ with respect to $(D, f)$. Otherwise, let $uv$ be a tight edge. By symmetry, we may assume that $e=(u,v)$ is a positive arc with $f(e)=r-1$. We say a vertex $w$ is reachable from $u$ if there is an oriented path $P$ from $u$ to $w$ such that for each arc $e'=(x,y)$ on $P$, $e'$ is either a positive (respectively, a negative) forward arc with $f(e) < r-1$  (respectively, $f(e') < \frac{r}{2}-1$), or a positive (respectively, a negative) backward arc with $f(e) > 1$  (respectively, $f(e') >\frac{r}{2}+1$). 
		As $(G, \sigma)$ has no tight cut with respect to $(D, f)$, $v$ is reachable from $u$. Let $P$ be the path from $u$ to $v$ defined as above.   
		Then $B=P+(v,u)$ is an oriented cycle (where $e=(v,u)$ is viewed as a backward edge). Choose a sufficiently small  $\epsilon > 0$, and let $f'(e')=f(e')+\epsilon$ for each forward arc $e'$ of $B$ and $f'(e')=f(e')-\epsilon$ for each   backward arc $e'$ of $B$, we obtain a   circular $r$-flow $f'$ in which $e$ is no longer a tight edge and no new tight edge is created.  This contradicts   the choice of $(D, f)$, and hence there is no tight edge  in $(G, \sigma)$ with respect to $(D, f)$. Then for a sufficiently small $\epsilon > 0$, $(D, \frac{1}{1+\epsilon}f)$ is   a circular $\frac{r}{1+\epsilon}$-flow in $(G, \sigma)$, a contradiction.  
		
		For the ``if'' part, we need to show that if $\Phi_c(G, \sigma) < r$, then there is a circular $r$-flow in $(G, \sigma)$ such that there is no tight cut. Assume $\Phi_c(G, \sigma) = r' < r$ and let $(D, f')$ be a non-negative circular $r'$-flow in $(G, \sigma)$. Let $f =\frac{r}{r'} f'$. Then  $(D, f)$ is a non-negative circular $r$-flow in $(G, \sigma)$ and $(G, \sigma)$ contains no tight edge with respect to $(D,f)$.  
	\end{proof}
	
	Assume that $\Phi_c(G, \sigma)=r$ and $(D, f)$ is a circular $r$-flow in $(G, \sigma)$. By Lemma~\ref{lem:tightcut}, there exists a tight cut $(X, X^c)$ with respect to $(D, f)$. Assume that in $(X, X^c)$, there are $s_1$ positive forward (i.e., from $X$ to $X^c$) arcs $e$ having $f(e)=r-1$, $s_2$ positive backward arcs $e$ having $f(e)=1$, $t_1$ negative forward arcs $e$ having $f(e)=\frac{r}{2}-1$, and $t_2$ negative backward arcs $e$ having $f(e)=\frac{r}{2}+1$. Applying Equation~(\ref{equ:Delta-f=0}) on the cut $(X, X^c)$, we have that
	\begin{equation}\label{equ:tightCut}
	s_1(r-1)+t_1(\frac{r}{2}-1)=s_2+t_2(\frac{r}{2}+1).
	\end{equation}
	Thus, we can determine the circular flow index $r$ of $(G, \sigma)$: 
	\begin{equation}\label{equ:FormulaFor r}
	r=\dfrac{{2(s_1+s_2+t_1+t_2)}}{2s_1+t_1-t_2}=\dfrac{2|(X, X^c)|}{2s_1+t_1-t_2}.
	\end{equation}
	
	The values of $s_1$, $s_2$, $t_1$ and $t_2$ are all bounded by the number of edges of $G$. Therefore, Formula~(\ref{equ:FormulaFor r}) limits the possible choices of $\Phi_c(G, \sigma)$ to a rational number whose numerator and denominator each is bounded by $2|E(G)|$. Thus we have the following result.
	
	\begin{theorem}\label{thm:inf is attained}
		For every finite signed graph $(G, \sigma)$ with no positive bridge, $$\Phi_c(G, \sigma)=\min \{\dfrac{p}{q}\mid (G, \sigma) \text{ admits a circular } \dfrac{p}{q}\text{-flow}, ~ 1\leq 2q \leq p\leq 2|E(G)|\}.$$
		In particular, $\Phi_c(G, \sigma)$ is a rational number.
	\end{theorem}
	
	Note that  by definition $\Phi_c(G, -)=2$ for any graph $G$.
	For signed graphs in general, it follows from the definition that  a circular $r$-flow in $G$ is also a circular $2r$-flow in $(G, \sigma)$ for any signature $\sigma$. 
	Thus we have the following proposition.
	\begin{proposition}\label{prop:2phi}
		For any graph $G$ without bridges and for any signature $\sigma$, $$\Phi_c(G, \sigma) \le 2 \Phi_c(G).$$
	\end{proposition}
	
	This upper bound is tight as below.

	\begin{lemma}\label{lem:FlowOfT_2(G)}
		Given a   graph $G$, let $T_2(G)$ be a signed graph obtained from $G$ by replacing each edge with a negative path of length $2$. Then $\Phi_c(T_2(G))=2\Phi_c(G)$.
	\end{lemma}
	\begin{proof}
		Assume that $\Phi_c(G)=r$ and $(D,f)$ is a  circular $r$-flow in $G$.  We define an orientation $D'$ of $T_2(G)$ based on $D$ as follows: for each edge $uv\in G$ and its replacement path $uwv$ in $T_2(G)$, if $(u,v)\in D$, then $(u,w), (w,v)\in D'$. Moreover, we define $f'(uw)=f'(wv)=f(uv)\in [1,r-1]$. Then it is easy to verify that $(D', f')$ is a circular $2r$-flow in $T_2(G)$ by definition. Thus $\Phi_c(T_2(G))\le 2r=2\Phi_c(G)$.
		
		For the other direction, given $s\geq 2$, assume that $\Phi(T_2(G))=2s$ and $T_2(G)$ admits a circular $2s$-flow $(D',f')$. We may further choose the orientation $D'$ such that for any $2$-thread $uwv$, either $(u,w),(w,v)\in D'$ or $(v,w),(w,u)\in D'$. Hence, we also have $f'(uw)=f'(vw)$ for any $2$-thread $uwv$. Now we define $(D, f)$ of $G$ as follows: we orient edge $uv$ the same as $2$-thread $uwv$ in $D'$ and let $f(uv)=f'(uw)=f'(vw)\in [1,2s-1]\cap ([0,s-1]\cup[s+1,2s))$. Then we have $f(uv)\in [1,s-1]\cup [s+1,2s-1]$ for any edge $uv\in E(G)$. By definition, $(D,f)$ is a circular modulo $s$-flow in $G$. By Theorem~\ref{thm:Equiavlences}, we have $\Phi_c(T_2(G))\ge 2\Phi_c(G)$. 
		This completes the proof.
	\end{proof}

	\section[Modulo orientation]{Modulo $\ell$-orientation and homomorphism to cycles}\label{sec:ModuloOrientation}
	
	In Jaeger's circular flow conjecture, the real numbers of the form $2+\frac{1}{k}$ play a special role. A graph $G$ is circular $(2+\frac{1}{k})$-colorable if and only if $G$ admits a homomorphism to the odd cycle $C_{\scriptscriptstyle 2k+1}$. While to determine whether a graph admits a homomorphism to an odd cycle is an interesting and difficult (NP-complete) problem   \cite{MSW81}, it is easy to decide whether a graph $G$ admits a homomorphism to an even cycle. In contrast, to decide whether a signed graph $\hat{G}$ admits a homomorphism to a negative even cycle is difficult and related to many challenging conjectures \cite{NPW22,NRS15,NWZ21}.
	The dual concept of admitting  homomorphisms to cycles is ``modulo $\ell$-orientation''. We propose the following generalization to signed graphs.
	
	\begin{definition}\label{def:ModuloOrientation}
		A signed graph $(G, \sigma)$ is said to be \emph{modulo $\ell$-orientable} if there exists an inversing-equivalent signature $\sigma'$ and an orientation $D$ on $G$ such that, with respect to $(G, \sigma')$,   for each vertex $v\in V(G)$,
		\begin{equation}\label{equ:moduloR-orientation}
		(\ell-1)(\overleftarrow{d^{+}}(v)-\overrightarrow{d^{+}}(v))=\overleftarrow{d^{-}}(v)-\overrightarrow{d^{-}}(v).
		\end{equation}
		Such an orientation $D$ is called a \emph{modulo $\ell$-orientation} on $(G, \sigma)$.
	\end{definition}
	
	\begin{observation}
		\label{obs-1}
		Assume $(G, \sigma)$ admits a modulo $\ell$-orientation.
		\begin{itemize}
			\item If $\ell$ is an odd number, then the left side of Equation~(\ref{equ:moduloR-orientation}), being multiplied by $\ell-1$, is an even number. Thus in the subgraph induced by the set of negative edges with respect to $\sigma'$, the difference of in-degree and out-degree is an even number. Hence, in this subgraph, the degree of each vertex is even. So we can apply some inversing on cycles to inverse all the edges into positive signs. So $(G, \sigma)$ is inversing equivalent to $(G, +)$.
			\item If $\ell$ is an even number, then we conclude that at each vertex the difference of in-degree and out-degree for positive edges and negative edges, with respect to $\sigma'$, is of the same parity. That means the total degree of each vertex is even. Hence $G$ is an even-degree graph. 
		\end{itemize}
	\end{observation}
	
	If restricted to graphs, the case of $\ell$ being even is easy and not interesting. However, in the framework of signed graphs, modulo $\ell$-orientation for even $\ell$ is as interesting as for odd $\ell$.

	\begin{theorem}\label{thm:mod k-orientation}
		A signed graph $(G, \sigma)$ admits a modulo $\ell$-orientation if and only if there is a partition $\{E_1, E_2,\ldots, E_{\ell}\}$ of $E(G)$ such that the following two conditions are satisfied: 
		\begin{enumerate}[label=(\roman*)]
			\item\label{E1} Each $E_i$ is the set of positive edges of a signature which is inversing equivalent to $\sigma$;
			\item\label{E2} There is an orientation on $G$ such that for each vertex $v$ and any pair $i, j\in \{1, \ldots, \ell\}$, 
			\begin{equation}\label{equ:EiPartition}	\overleftarrow{d_{\scriptscriptstyle E_i}}(v)-\overrightarrow{d_{\scriptscriptstyle E_i}}(v)= \overleftarrow{d_{\scriptscriptstyle E_j}}(v)-\overrightarrow{d_{\scriptscriptstyle E_j}}(v).
			\end{equation}
		\end{enumerate}
	\end{theorem}

	\begin{proof}
		Suppose that $E(G)$ is partitioned into $E_1, E_2,\ldots, E_{\ell}$, each $E_i$ being the set of positive edges of the signature $\sigma_i$ where each $\sigma_i$ is inversing equivalent to $\sigma$, and  $D$ is an orientation satisfying Condition~\ref{E2} of the theorem. We then add up Equations~(\ref{equ:EiPartition}) for the pairs $(1, j)$, $j=2,3, \ldots, \ell$, and we obtain Equation~(\ref{equ:moduloR-orientation}) for a modulo $\ell$-orientation on $(G, \sigma)$, with respect to the inversing-equivalent signed graph $(G, \sigma_1)$.

		Conversely, suppose that, with respect to an inversing-equivalent signature $\sigma'$, the signed graph $(G, \sigma')$ admits a modulo $\ell$-orientation $D$. 
		If at a vertex $v$ of $G$ there are both incoming and outgoing edges of the same sign with respect to the signature $\sigma'$, then let $uv$ and $vw$ be such a pair of an incoming edge and an outgoing edge. We delete these two edges and add the oriented edge $uw$ with the same sign as $vw$. We repeat this operation until there are no such pair of edges. Let $(G_1, \sigma')$ be the resulting signed graph and let $D'$ be the resulting orientation on $(G_1, \sigma')$. 
		It follows from the construction that 
		\begin{enumerate}
			\item $V(G_1)=V(G)$ (possibly $G_1$ contains isolated vertices),
			\item $\forall v \in V(G)$, $\overleftarrow{d^{+}_{\scriptscriptstyle G}}(v)-\overrightarrow{d^{+}_{\scriptscriptstyle G}}(v)=\overleftarrow{d^{+}_{\scriptscriptstyle G_{\scriptscriptstyle 1}}}(v)-\overrightarrow{d^{+}_{\scriptscriptstyle G_{\scriptscriptstyle 1}}}(v)$ and $ \overleftarrow{d^{-}_{\scriptscriptstyle G}}(v)-\overrightarrow{d^{-}_{\scriptscriptstyle G}}(v)=\overleftarrow{d^{-}_{\scriptscriptstyle G_{\scriptscriptstyle 1}}}(v)-\overrightarrow{d^{-}_{\scriptscriptstyle G_{\scriptscriptstyle 1}}}(v).$
		\end{enumerate}
		
		Thus $D'$ is a modulo $\ell$-orientation on $(G_1,\sigma')$, i.e., at each vertex $v$, $(\ell-1)d^{+}_{\scriptscriptstyle G_{\scriptscriptstyle 1}}(v)= d^{-}_{\scriptscriptstyle G_{\scriptscriptstyle 1}}(v)$. We  split each vertex $v$ into $d^{-}(v)$ copies $v_1,v_2, \ldots, v_{d^{-}(v)}$ where each $v_i$ takes one positive edge incident to $v$ and $\ell-1$ negative edges incident to $v$. Let $(G',\sigma'')$ be the resulting signed graph together with inherited orientation $D'$. Note that $(G',\sigma'')$ (together with $D'$) is an $\ell$-regular graph oriented in a way that each vertex is either a source or a sink and where each vertex has exactly one positive edge incident to it. Thus first of all the underlying graph is a bipartite graph with source vertices forming one part and sink vertices forming the other part. Secondly, the set of positive edges, denoted by $E'_1$, form a perfect matching. We now consider the $(\ell-1)$-regular subgraph of $G'$ obtained by removing edges of $E'_1$. By K\"onig's theorem, we can partition the edges of this $(\ell-1)$-regular bipartite graph into $\ell-1$ perfect matchings, denoted by $E_2', E_3', \ldots, E_{\ell}'$.  Note that there is a clear correspondence between the edges of $G$ and $G'$. Moreover, with respect to this natural correspondence between the edges, the orientation on the edges will remain the same. So the corresponding edge sets $E_1, E_2,\ldots, E_{\ell}$ in $G$ form the partition as required.
	\end{proof}
	
	We note that for $\ell=2k+1$, Definition~\ref{def:ModuloOrientation} is equivalent to the classic definition of modulo $(2k+1)$-orientation of graphs in the literature, in which  a modulo $(2k+1)$-orientation is defined as an orientation such that for each vertex the in-degree is congruent to out-degree modulo $2k+1$. 
	As observed before, when $\ell=2k+1$, we only consider signed graphs $(G, +)$, and when $\ell =2k$ is even, we only need to consider signed Eulerian graphs. Lemma~\ref{lem:Jaeger} was proved in \cite{J88}.
	
	\begin{lemma}\label{lem:Jaeger}{\rm \cite{J88}}
		A graph admits a circular $\frac{2k+1}{k}$-flow if and only if it admits a modulo $(2k+1)$-orientation.
	\end{lemma}
	
	As an extension of Lemma~\ref{lem:Jaeger} to the case that $\ell=2k$ is even, we show below that a signed Eulerian graph admits a circular $\frac{4k}{2k-1}$-flow if and only if it admits a modulo $2k$-orientation. In fact, we obtain some more equivalent properties in the following lemma.

	\begin{lemma}\label{lem:EulerianEquivalence}
		Let $\hat{G}$ be a signed Eulerian graph. Then the following statements are equivalent.
		\begin{enumerate}[label=(\arabic*)]
			\item $\hat{G}$ admits a circular $\frac{4k}{2k-1}$-flow.
			\item $\hat{G}$ admits a modulo $4k$-flow $(D, f)$ such that $f(e)\in\{2k-1, 2k+1\}$ for each positive edge $e$, and $f(e)\in\{-1, 1\}$ for each negative edge $e$.
			\item $\hat{G}$ admits an orientation $D$ such that $\overleftarrow{d_{\scriptscriptstyle D}}(v)-\overrightarrow{d_{\scriptscriptstyle D}}(v)\equiv 2k\cdot d^+_{\scriptscriptstyle \hat{G}}(v)~\pmod{4k}$ for each vertex $v\in V(G)$.
			\item $\hat{G}$ admits a modulo $2k$-orientation. 
		\end{enumerate}
	\end{lemma}
	
	\begin{proof}
		$(1) \Rightarrow (2)$. Assume that $\hat{G}$ is a signed Eulerian graph and $(D, \varphi)$ is a circular $\frac{4k}{2k-1}$-flow in $\hat{G}$. Then $\varphi(e)\in \{2k-1, 2k, 2k+1\}$ for each positive edge $e$ and $\varphi(e)\in \{4k-1, 0, 1\}$ for each negative edge $e$. Since $G$ is Eulerian,  the set $E'=\{e \in E(G): \varphi(e) \in \{0,2k\}\}$ induces an even-degree subgraph. Partition $E'$ into edge-disjoint cycles. For each cycle $C$ in this partition, we modify $\varphi$ on $C$ as follows:  If $C$ is a directed cycle, then let $\varphi'(e) = \varphi(e)+1$ for each edge $e \in C$. Otherwise, the cycle $C$ is divided into segments $P_1,P_2,\ldots, P_{2t}$, where each $P_i$ is a directed path and for each $i$, one of $P_i$ or $P_{i+1}$ is a forward directed path and the other is a backward directed path.  Let $\varphi'(e)=\varphi(e)+1$ for $e \in P_1,P_3, \ldots, P_{2t-1}$ and  $\varphi'(e)=\varphi(e)-1$ for $e \in P_2,P_4, \ldots, P_{2t}$. 
		For edges $e \notin E'$, let $\varphi'(e)=\varphi(e)$. Then $(D, \varphi')$ is a modulo $4k$-flow in $\hat{G}$ where $\varphi'(e) \in \{2k+1,2k-1\}$ for each positive edge $e$ and $\varphi'(e) \in \{-1,1\}$ for each negative edge $e$. 
		
		\medskip 
		
		$(2) \Rightarrow (3)$. Assume that $\hat{G}$ admits a modulo $4k$-flow $(D, f)$ such that $f(e)\in\{2k-1, 2k+1\}$ for each positive edge $e$, and $f(e)\in\{-1, 1\}$ for each negative edge $e$. We define a mapping $f':E(G)\to \{0, 2k\}$ as follows: $f'(e)=2k$ for each positive edge $e$, and $f'(e)=0$ for each negative edge $e$. As $\overleftarrow{d^+_{\scriptscriptstyle D}}(v)-\overrightarrow{d^+_{\scriptscriptstyle D}}(v)\equiv d^+_{\scriptscriptstyle \hat{G}}(v)~\pmod{2}$, we have $\partial_{\scriptscriptstyle D}f'(v)\equiv 2k\cdot d^+_{\scriptscriptstyle \hat{G}}(v)~\pmod {4k}$. Let $g=f+f'$. It is easy to observe that $g(e)~\pmod {4k}\in \{-1,1\}$ and as $\partial_{\scriptscriptstyle D}f(v)\equiv 0~\pmod {4k}$, we have $\partial_{\scriptscriptstyle D}g(v)\equiv 2k\cdot d^+_{\scriptscriptstyle \hat{G}}(v)~\pmod {4k}$. We define a new orientation $D'$ based on $D$ and $g$ as follows: an edge is oriented in $D'$ the same as in $D$ if $g(e)\equiv 1~\pmod {4k}$ and is orientated in $D'$ opposite as in $D$ if $g(e)\equiv -1~\pmod {4k}$. Such an orientation $D'$ satisfies that $\overleftarrow{d_{\scriptscriptstyle D'}}(v)-\overrightarrow{d_{\scriptscriptstyle D'}}(v)\equiv 2k\cdot d^+_{\scriptscriptstyle \hat{G}}(v)~\pmod{4k}$ for each vertex $v\in V(G)$.

		\medskip
		
		$(3) \Rightarrow (4)$. Let $D$ be an orientation on $\hat{G}=(G,\sigma)$ such that $\overleftarrow{d_{\scriptscriptstyle D}}(v)-\overrightarrow{d_{\scriptscriptstyle D}}(v)\equiv 2k\cdot d^+_{\scriptscriptstyle \hat{G}}(v)~\pmod{4k}$. We may regard $D$ as a directed graph. Following the methods of the proof of Theorem~\ref{thm:mod k-orientation}, we build a new digraph as follows: First we lift pair of arcs $(u,v), (v,w)$, and continue this process until we obtain a digraph $D_1$ where each vertex is either a source or a sink. We note that $D_1$ might not be unique and depends on how we apply the process. However, we have $\overleftarrow{d_{\scriptscriptstyle D_{1}}}(v)-\overrightarrow{d_{\scriptscriptstyle D_1}}(v)= \overleftarrow{d_{\scriptscriptstyle D}}(v)-\overrightarrow{d_{\scriptscriptstyle D}}(v)$ for any vertex $v\in V(D_1)$. Thus in the underlying graph of $D_1$, we conclude that the degree of each vertex is a multiple of $2k$. For the final step of building the new digraph, we split vertices of $D_1$ so to have a digraph $D_2$ whose underlying graph is a $2k$-regular bipartite graph. Being a bipartite regular graph, by K\"onig's theorem, we can partition the edges of this $2k$-regular bipartite graph into $2k$ perfect matchings. Let $E_1, E_2, \ldots, E_{2k}$ be such a partition. By identifying the arc set of $D_1$ with $D_2$, we note that in the process of constructing $D_2$ from $D$, each edge of $D$ has a unique corresponding edge in $D_2$. More precisely, edges of $D_2$ correspond to a partition of edges of $D$ into directed paths. This leads to an edge partition $E_1, E_2, \ldots, E_{2k}$ of $\hat{G}$. For each $i\in [2k]$, we define $\sigma_i$ to be the signature on $G$ where positive edges are those in $E_i$. By Theorem~\ref{thm:mod k-orientation}, it suffices to show that (i) $\sigma_i$ is inversing equivalent to $\sigma$, and (ii) under the orientation $D$ of $\hat{G}$,  $\overleftarrow{d_{\scriptscriptstyle E_i}}(v)-\overrightarrow{d_{\scriptscriptstyle E_i}}(v)= \overleftarrow{d_{\scriptscriptstyle E_j}}(v)-\overrightarrow{d_{\scriptscriptstyle E_j}}(v)$.

		Let $V^o(\hat{G})=\{v\in V(\hat{G})\mid d^+_{\scriptscriptstyle \hat{G}}(v) \text{~is odd}\}$. By Lemma~\ref{lem:Cycle-Switch-Equiv} and noting that $\hat{G}$ is Eulerian, to show that two signatures $\sigma$ and $\sigma'$ are inversing equivalent, it suffices to prove that $V^o(G, \sigma)=V^o(G, \sigma')$. Note that for each vertex $w$ in $V^o(G, \sigma)$, we have $\overleftarrow{d_{\scriptscriptstyle D}}(w)-\overrightarrow{d_{\scriptscriptstyle D}}(w)\equiv 2k~\pmod{4k}$. Hence, it has been split into odd number of vertices in $D_2$, and each of whom is exactly in one edge of $E'_i$, $i\in [2k]$. As each edge of $E'_i$ is represented by an edge-disjoint path in $E_i$ and $\hat{G}$ is Eulerian, in $(G, \sigma_i)$, also there is an odd number of positive edges incident to $w$. Hence $V^o(G, \sigma)=V^o(G, \sigma_i)$ for each $i\in [2k]$. This completes the proof of (i). Since in $D_1$ each vertex is either a sink or source, each of it contains the same number of edges from each of $E_i$. The edge pair $(u,v), (v,w)$ which has been lifted contributes $0$ to the term ``$\overleftarrow{d_{\scriptscriptstyle E_i}}(v)-\overrightarrow{d_{\scriptscriptstyle E_i}}(v)$" as they are in the same class $E'_i$ in $D_1$. So the value $\overleftarrow{d_{\scriptscriptstyle E_i}}(v)-\overrightarrow{d_{\scriptscriptstyle E_i}}(v)$ is the number of edges of $D_1$ in $E_i$. Hence, under the orientation $D$ of $\hat{G}$, $\overleftarrow{d_{\scriptscriptstyle E_i}}(v)-\overrightarrow{d_{\scriptscriptstyle E_i}}(v)= \overleftarrow{d_{\scriptscriptstyle E_j}}(v)-\overrightarrow{d_{\scriptscriptstyle E_j}}(v)$. 
		
		\medskip
		
		$(4) \Rightarrow (1)$. Assume that, with respect to an inversing-equivalent signature, $\hat{G}$ admits an orientation $D$ such that for each vertex $v$,
		$(2k-1)(\overleftarrow{d^{+}}(v)-\overrightarrow{d^{+}}(v))=\overleftarrow{d^{-}}(v)-\overrightarrow{d^{-}}(v).$ Let $f$ be defined as follows: $f(e)=2k-1$ if $e$ is a positive edge and $f(e)=-1$ if $e$ is a negative edge. Based on $(D, f)$, for each vertex $v$, we have    
		$$\partial_{\scriptscriptstyle D}f(v) =(2k-1)\cdot(\overleftarrow{d^+_{\scriptscriptstyle D}}(v)-\overrightarrow{d^+_{\scriptscriptstyle D}}(v))+(-1)\cdot(\overleftarrow{d^-_{\scriptscriptstyle D}}(v)-\overrightarrow{d^-_{\scriptscriptstyle D}}(v))=0.$$
		Such a pair $(D, f)$ is a modulo $(4k, 2k-1)$-flow of $\hat{G}$ and thus $\hat{G}$ admits a circular  $\frac{4k}{2k-1}$-flow.
	\end{proof}
	
	Observe that a signed Eulerian planar graph $(G, \sigma)$ admits a modulo $2k$-orientation if and only if its dual signed graph $(G^*, \sigma^*)$ admits a homomorphism to $C_{ \!\scriptscriptstyle -2k}$. However, a translation of this theorem to a homomorphism theorem (stated below) holds for all signed graphs. Considering the similarity of the proofs, we omit its proof. 
	
	\begin{theorem}\label{thm:Packing-Vrsus-C_k-coloring}
		A signed graph $(G, \sigma)$ admits a homomorphism to $-C_{\CK}$ if and only if there is a partition $E_1,E_2, \ldots E_k$ of edges of $G$ such that the following hold:
		\begin{enumerate}[label=(\roman*)]
			\item Each $E_i$ is the set of positive edges of a signature which is switching equivalent to $\sigma$.
			\item There is an orientation $D$ on $G$ with the following property: For each cycle $C$ of $G$ there is a constant $w_{_{C}}$ such that the difference between the number of the forward and the number of the backward edges of $C$ that are in $E_i$, for any $i\in [k]$, is $w_{_{C}}$.
		\end{enumerate} 
	\end{theorem}
	
	A visual version of this theorem is as follows: Consider $-C_{\CK}$ as a directed cycle with one positive edge. Let $\varphi$ be a homomorphism of $(G, \sigma)$ to $-C_{\CK}$ and let $D$ be the orientation on $G$ induced by this homomorphism. Given a cycle $C$ of $(G, \sigma)$, the difference of forward and backward edges of $C$ then must be a multiple of $k$, say $w_{_{C}}\!\cdot \! k$ where $w_{_{C}}$ represents the number of times the cycle $C$ is winded around $-C_{\CK}$ by $\varphi$. Moreover, the sign of $w_{_{C}}$ represents if this winding is in the clockwise direction or the anticlockwise direction. The value of $w_{_C}$ can be computed using just one of the edges of $-C_{\CK}$, say $e$. More precisely, $w_{_C}$ is the difference of the number of forward edges in $C$ that are mapped to $e$ and the number of backward edges in $C$ that are mapped to $e$.
	
	The heart of this theorem then is that: Finding an orientation which is consistent with the above conclusion is enough to build a mapping of $(G,\sigma)$ to $-C_{\CK}$. A slightly simpler condition that implies the same result is as follows: Given $(G,\sigma)$ there is a switching-equivalent signature $\sigma'$ and an orientation $D$ such that in every cycle $C$ of $G$, with respect to $\sigma'$ and $D$, the difference of forward and backward on the negative edges is $(k-1)$-times of the difference of forward and backward on the positive edges.

	\section{Edge connectivity and circular flow index}\label{sec:Main}

	Seymour's $6$-flow theorem states that every $2$-edge-connected graph has a circular $6$-flow. By Proposition \ref{prop:2phi} and Lemma \ref{lem:FlowOfT_2(G)}, Seymour's $6$-flow theorem can be restated in signed graphs as below.
	
	\begin{theorem}{\em [Seymour's $6$-flow theorem restated]}\label{thm:12-flow}
		Every $2$-edge-connected signed graph admits a circular $12$-flow.
	\end{theorem}
	
	Similarly, Tutte's $5$-flow conjecture in graphs can be, equivalently, stated as:
	
	\begin{conjecture}{\em [Tutte's $5$-flow conjecture restated]}\label{conj:10-flow}
		Every $2$-edge-connected signed graph admits a circular $10$-flow.
	\end{conjecture}
	
	For graphs of higher connectivity, stronger results can be obtained using the notion of group connectivity, introduced by Jaeger, Linial, Payan, and Tarsi \cite{JLPT92}.
	
	\begin{definition}
		Given a graph $G$, a mapping $\beta:V(G) \to \mathbb{Z}_k$ is called a \emph{$\mathbb{Z}_k$-boundary} if  $\sum\limits_{v\in V(G)} \beta(v)\equiv 0~\pmod k.$ A graph $G$ is said to be \emph{$\mathbb{Z}_k$-connected} if for every $\mathbb{Z}_k$-boundary $\beta$, there is an orientation $D$ on $G$ and a mapping $f: E(G) \to \mathbb{Z}_k$ such that for each vertex $v\in V(G)$, $\partial_{\scriptscriptstyle D} f(v) \equiv \beta(v)~\pmod{k}$.
	\end{definition}
	
	Proposition \ref{pro:Z_kconnectivity} and Theorem \ref{thm:Z_6Z_4} below were proved in \cite{JLPT92}. 
	
	\begin{proposition}\label{pro:Z_kconnectivity} {\rm \cite{JLPT92}}
		Let $G$ be a connected graph with an arbitrary orientation $D$. The following statements are equivalent:
		\begin{itemize}
			\item $G$ is $\mathbb{Z}_k$-connected.
			\item For any function $g: E(G)\to \mathbb{Z}_k$, there exists a modulo $k$-flow $(D, f)$ such that for each edge $e$ of $G$, $f(e)\not\equiv g(e)~\pmod k$.
		\end{itemize}
	\end{proposition}

	\begin{theorem} \label{thm:Z_6Z_4} {\rm \cite{JLPT92}}
		Every $3$-edge-connected graph is $\mathbb{Z}_6$-connected, and 
		every graph with two edge-disjoint spanning trees is $\mathbb{Z}_4$-connected. In particular, every $4$-edge-connected graph is $\mathbb{Z}_4$-connected.
	\end{theorem}
	
	Now we use Theorem \ref{thm:Z_6Z_4}  to derive results concerning circular flow in signed graphs.

	\begin{theorem}\label{thm:3-edgeconnected6-flow}
		Every $3$-edge-connected signed graph admits a circular $6$-flow, and every $4$-edge-connected signed graph admits a circular $4$-flow.
	\end{theorem}
	\begin{proof}
		Assume $(G, \sigma)$ is a $3$-edge-connected signed graph. 
		Let $g: E(G)\to \mathbb{Z}_6$ to be defined as 
		$$g(e) = \begin{cases} 0, &\text{ if $e$ is a positive edge,} \cr 3, &\text{  if $e$ is a negative edge.} \cr \end{cases}$$ 
		By 	Proposition \ref{pro:Z_kconnectivity} and Theorem \ref{thm:Z_6Z_4}, $G$ admits a modulo $6$-flow $(D,f)$ such that $f(e)\not\equiv g(e)~\pmod 6$ for each edge $e$.  Then $(D, f)$ is a circular modulo $6$-flow in $(G, \sigma)$.
		
		The other half of the theorem is proved in the same way.  
	\end{proof}
	
	\begin{theorem}\label{thm:3spanningtree<4flow}
		For any signed graph $(G, \sigma)$ that contains $3$ edge-disjoint spanning trees, we have $\Phi_c(G, \sigma)<4$. In particular, for every $6$-edge-connected signed graph $(G, \sigma)$, we have $\Phi_c(G, \sigma)<4$.
	\end{theorem}

	\begin{proof}
		Let $T_1, T_2$ and $T_3$ be three edge-disjoint spanning trees of the underlying graph $G$.  By Lemma~\ref{lem:SpanningTree}, we consider the inversing equivalent signed graph $(G, \sigma')$ where all the negative edges are in $T_1$. By Lemma \ref{lem:tightcut}, it suffices to construct a circular $4$-flow $(D, f)$ of $(G, \sigma')$ that has no tight cut.
		
		Let $D$ be an orientation on $G$ and let $g: E(G)\to \mathbb{Z}_4$ be defined as
		$$g(e) = \begin{cases} 0, &\text{ for each positive edge $e\in E(T_1)\cup E(T_2)$,} \cr 2, &\text{ otherwise.} \cr \end{cases}$$ 
		Let $\beta^*: V(G)\to \mathbb{Z}_4$ be the map satisfying that 
		$\beta^*(v)\equiv \partial_{\scriptscriptstyle D}g(v)~\pmod 4$. Then  $\beta^*$ is a $\mathbb{Z}_4$-boundary. Let $H=T_1 \cup T_2$. By Theorem~\ref{thm:Z_6Z_4}, $H$ is $\mathbb{Z}_4$-connected, and hence there exists a mapping $f^*: E(H) \to \mathbb{Z}_4\setminus\{0\}$ such that $\partial_{\scriptscriptstyle D}f^*(v)\equiv \beta^*(v)~\pmod 4$ for each vertex $v\in V(H)$.   Extend   $f^*$ to the whole graph $G$ by letting $f^*(e)=0$ for   $e\in E(G\setminus H)$.
		Let $f=f^*-g$. Then $\partial_{\scriptscriptstyle D}f(v)\equiv 0~\pmod 4$ for every vertex $v\in V(G)$ and  $(D, f)$ is a circular modulo $4$-flow in $(G, \sigma')$, since $f(e)\neq 2$ for each negative edge $e\in T_1$ and $f(e)\neq 0$ for each positive edge $e$.  As $f(e) =2$  for each edge $e \in T_3$ and any cut contains an edge in $T_3$, $(G, \sigma')$ contains no tight cut with respect to $(D,f)$.
		
		By Nash-Williams-Tutte Theorem~\cite{N61,T61}, $2k$-edge-connected graphs have $k$ edge-disjoint spanning trees. Thus  $\Phi_c(G, \sigma)<4$ for every $6$-edge-connected signed graph $(G, \sigma)$.
	\end{proof}

	It was proved in \cite{LTWZ13} that $6$-edge-connected graph is $\mathbb{Z}_3$-connected. However, the argument in the proof of    Theorem~\ref{thm:3-edgeconnected6-flow} cannot be applied to show that every $6$-edge-connected  signed graph admits a circular $3$-flow.
	It is because our definition of circular (modulo) $3$-flow in signed graphs is quite different from the notion of $\mathbb{Z}_3$-connected graphs. 
	
	Before moving to the higher edge-connectivity, we need the following definitions and results concerning orientations with boundaries from \cite{LWZ20}.
	
	\begin{definition}
		(1)	Given a graph $G$, a {\em parity-compliant $2k$-boundary} is a mapping $\beta:  V(G) \to \{0, \pm 1, \ldots, \pm k\}$ satisfying two conditions:
		
		\begin{itemize}
			\item $\sum\limits_{v\in V(G)} \beta(v)\equiv 0~\pmod {2k}$, and
			\item for every vertex $v\in V(G)$, $ \beta(v)\equiv d(v)~\pmod {2}.$
		\end{itemize}
		
		(2)	Given a parity-compliant $2k$-boundary $\beta$, an orientation $D$ on $G$ is called a {\em  $(\mathbb{Z}_{2k},\beta)$-orientation} if for every vertex $v \in V (G)$, $$\overleftarrow{d_{ \scriptscriptstyle D}}(v)-\overrightarrow{d_{\scriptscriptstyle D}}(v) \equiv \beta(v) ~\pmod {2k}.$$
	\end{definition}
	
	Given a {\em parity-compliant $2k$-boundary}, for a subset $A\subset V(G)$, we define $\beta(A)\in \{0, \pm 1, \ldots, \pm k\}$ by $\beta(A)\equiv \sum\limits_{v\in A}\beta(v)~\pmod {2k}.$ 
	
	\begin{definition}\label{def:D_e}
		Given a (partial) orientation $D$ of a graph $G$ and an arc $e=(x, y)$ in $D$, let $D_e$ be the (partial) orientation obtained from $D$ by flipping the arc $e$, that is to remove $(x, y)$ and to add $(y, x)$. Similarly, given a parity-compliant $2k$-boundary $\beta$ and already oriented edge $(x, y)$ we define $\beta_{e}$ as follows: 
		$$\beta_{e}(x)=\beta(x) - 2, ~~\beta_{e}(y)=\beta(y) + 2, ~~\text{ and for every other vertex } v, ~ \beta_e(v)=\beta(v).$$ 
	\end{definition}
	
	It is easy to see that $\beta_{e}$ is also a parity-compliant $2k$-boundary of $G$. Here the calculations are done in $\mathbb{Z}_{2k}$ whose presentation of elements will be understood from the context. Normally, as in the previous definition, we will present them by $\beta:  V(G) \to \{0, \pm 1, \ldots, \pm k\}$ where $\pm k$ present the same element but it is more suitable to allow both presentation. With these definitions, we have the following immediate observation.
	
	\begin{observation}\label{obs:D,B-->D_e,B_e}
		Given a signed graph $(G, \sigma)$ and a parity-compliant $2k$-boundary function $\beta$, an orientation $D$ on $G$ is a $(\mathbb{Z}_{2k},\beta)$-orientation if and only if $D_e$ is a $(\mathbb{Z}_{2k},\beta_e)$-orientation.		
	\end{observation} 
	
	The following theorem of \cite{LTWZ13,LWZ20} is one of the key elements of our proofs. It basically claims that subject to some connectivity condition, some partial orientations satisfying basic necessary conditions can be extended to a  $(\mathbb{Z}_{2k},\beta)$-orientation of the full graph.

	\begin{theorem}\label{evenpartialextending-old} {\rm \cite{LTWZ13,LWZ20}}
		Let $G$ be a graph and let $\beta$ be a parity-compliant $2k$-boundary of $G$, where $k\ge 3$. Let $z$ be a vertex of $V(G)$ such that $d(z)\le 2k-2 + |\beta(z)|$.  Assume that $D_{z}$ is an orientation on $E(z)$ (edges incident to $z$) which achieves boundary $\beta(z)$ at $z$. Let $V_0=\{v\in V(G)\setminus \{z\} \mid \beta(v)=0\}$. If $V_0\neq \emptyset$, then let $v_0$ be a vertex of $V_0$ with the smallest degree. Assume that $d(A)\ge 2k-2 + |\beta(A)|$ for any $A\subset V(G)\setminus\{z\}$ with $A\neq \{v_0\}$ and $|V(G)\setminus A|>1$. Then the partial orientation $D_{z}$ can be extended to a $(\mathbb{Z}_{2k},\beta)$-orientation on the entire graph $G$.
	\end{theorem}

	\begin{theorem}\label{THM:LWZ20}{\em \cite{LWZ20}}
		Let $G$ be a $(3k-3)$-edge-connected graph, where $k\ge 3$. Then for any parity-compliant $2k$-boundary $\beta$ of $G$, $G$ admits a $(\mathbb{Z}_{2k},\beta)$-orientation.
	\end{theorem}
	
	Recall that $d_{\scriptscriptstyle \hat{G}}^+(v)$ (or simply $d^+(v)$ when the graph is clear from the context) denotes the number of positive edges incident to $v$ in the signed graph $\hat{G}$.

	\begin{theorem}\label{thm:orientation}
		Given positive integers $p$ and $q$ satisfying that $p\geq q$, a signed graph $\hat{G}$ admits a $(2p,q)$-flow if and only if the graph $(2p-2q){G}$ admits a $(\mathbb{Z}_{4p}, \beta)$-orientation with $\beta(v)\equiv 2p\cdot d^+_{\scriptscriptstyle \hat{G}}(v)~\pmod{4p}$ for each vertex $v\in V(G)$. 
	\end{theorem}
	
	\begin{proof}
		Assume that $D$ is a $(\mathbb{Z}_{4p}, \beta)$-orientation on $(2p-2q){G}$ where $\beta(v)\equiv 2p\cdot d^+_{\scriptscriptstyle \hat{G}}(v)~\pmod{4p}$ for each vertex $v\in V(G)$. This in particular means that $\beta$ is a parity-compliant $4p$-boundary. Observe that if we define $h(e)=1$ for each edge $e$ of $(2p-2q){G}$, then we have $\partial_{\scriptscriptstyle D} h(v)\equiv \beta(v) ~\pmod {4p}$.  	
		
		Let $D'$ be an orientation on $G$. For each $e\in E(G)$, let $[e]$ denote the set of  the corresponding  $2p-2q$ parallel edges in $(2p-2q){G}$. Let $I_e: [e]\to \{\pm 1\}$ be defined as follows: $I_e(e')=1$ if $e'\in [e]$ in $D$ is oriented as the same as $e$ in $D'$ and $I_e(e')=-1$ otherwise. Then we define a mapping $f_{_{I}}: E(G) \to \mathbb{Z}_{4p}$ as follows: $$f_{_{I}}(e)=\sum\limits_{e_i \in [e]}I_e(e_i).$$ Note that for each edge $e\in E(G)$, $f_{_{I}}(e)\in \{-(2p-2q), \ldots, -2, 0, 2, \ldots,2p-2q\}$, i.e., $|f_{_{I}}(e)|$ is even and it satisfies that $|f_{_{I}}(e)|\le 2p-2q$, and $(D', f_{_{I}})$ in $G$ satisfies that for each vertex $v\in V(G)$, $\partial_{\scriptscriptstyle D'}f_{_{I}}(v)\equiv \beta(v)~\pmod {4p}$. Next we define another mapping $g: E(G)\to \mathbb{Z}_{4p}$ as follows: $g(e)=0$ if $e$ is a negative edge and $g(e)=2p$ if $e$ is a positive edge. Then, for each $v\in V(G)$, we have $\partial_{\scriptscriptstyle D'} g(v)\equiv 2p\cdot d^+(v)~\pmod{4p}$.
		
		Let $f=f_{_{I}}+g$. Then $f: E(G)\to \mathbb{Z}_{4p}$ satisfies the following conditions: for each positive edge $e$, $f(e)=f_{_{I}}(e)+2p \in \{2q, 2q+2,\ldots, 4p-2q\}$ and for each negative edge $e$, $f(e)=f_{_{I}}(e)\in \{-(2p-2q), \ldots, -2, 0, 2, \ldots,2p-2q\}$. Furthermore, considering the orientation $D'$ on $G$, $$\partial_{\scriptscriptstyle D'} f(v)=\partial_{\scriptscriptstyle D'} f_{_{I}}(v)+\partial_{\scriptscriptstyle D'} g(v)\equiv \beta(v)+2p\cdot d^+(v)\equiv 0~\pmod{4p}.$$ Hence, $(D', f)$ is a modulo $(4p, 2q)$-flow in $\hat{G}$. By Theorem~\ref{thm:Equiavlences}, $\hat{G}$ admits a $(4p, 2q)$-flow and hence a $(2p,q)$-flow.  
		
		By reversing the process above, one can build a $(\mathbb{Z}_{4p}, \beta)$-orientation on the graph $(2p-2q){G}$ with $\beta(v)\equiv 2p\cdot d^+_{\scriptscriptstyle \hat{G}}(v)~\pmod{4p}$ from a given $(2p,q)$-flow. 
	\end{proof}

	We are now ready to give our main results about the upper bounds on the circular flow indices of signed graphs based on the edge connectivity of the underlying graphs. To prove the theorem, rather than study the circular flow in signed graphs directly, we apply Theorem~\ref{thm:orientation} to study an orientation property of $\alpha G$ for some choice of $\alpha$. 
	
	\begin{theorem}\label{thm:mainFlow}
		Given a signed graph $(G, \sigma)$ and an integer $k\geq2$, the following claims hold.
		\begin{enumerate}[label=(\arabic*)]
			\item\label{F-1} If $G$ is $(3k-1)$-edge-connected, then $\Phi_c(G, \sigma) \leq \dfrac{2k}{k-1}$.
			\item\label{F0} If $G$ is $3k$-edge-connected, then $\Phi_c(G, \sigma) < \dfrac{2k}{k-1}$.
			\item\label{F+1} If $G$ is $(3k+1)$-edge-connected, then $\Phi_c(G, \sigma) \leq \dfrac{4k+2}{2k-1}$.
			
		\end{enumerate}
	\end{theorem}

	\begin{proof}
		Let $(G, \sigma)$ be a signed graph and let $D$ be an orientation on $G$. We define a mapping $\beta_{_{\ell}}: V(G)\to \{0, 2\ell\}$ satisfying the following:  
		$$\beta_{_{\ell}}(v)\equiv 2\ell(\overleftarrow{d^{+}}(v)-\overrightarrow{d^{+}}(v))~\pmod{4\ell}.$$
		Note that $\sum\limits_{v\in V(G)}\beta_{_{\ell}}(v)\equiv 0~\pmod{4\ell}$. Moreover, $\beta_{_{\ell}}(v)=2\ell$ if $d^+(v)$ is odd and $\beta_{_{\ell}}(v)=0$ otherwise. Thus $\beta_{_{\ell}}(v)\equiv 2\ell\cdot d^+(v)~\pmod{4\ell}$.

		\medskip    
		\noindent    
		{\ref{F-1}.} To prove that $(G, \sigma)$ admits a $({2k},{k-1})$-flow, by Theorem~\ref{thm:orientation}, it would be enough to show that $2G$ admits a $(\mathbb{Z}_{4k}, \beta_{_{k}})$-orientation. To this end, we must first verify that $\beta_{_{k}}$ is a parity-compliant $4k$-boundary of $2G$. That is because $\beta_{_{k}}(v)$ is an even value for each vertex $v$ of $G$ and in $2G$ every vertex is of even degree. 	
		To get the required orientation on $2G$ we apply Theorem~\ref{THM:LWZ20}, noting that $2G$ is a $(6k-2)$-edge-connected graph.

		The proof of \ref{F+1} is quite similar and we provide this proof before proving \ref{F0}.
		
		\medskip    
		\noindent 	
		{\ref{F+1}.} To prove that $(G, \sigma)$ admits a $({4k+2},{2k-1})$-flow, by Theorem~\ref{thm:orientation}, it would be enough to show that $4G$ admits a $(\mathbb{Z}_{8k+4}, \beta_{_{2k+1}})$-orientation. That $\beta_{_{2k+1}}$ is a parity-compliant $(8k+4)$-boundary of $4G$ is implied similar to the previous case. 	
		To get the required orientation on $4G$ once again we apply Theorem~\ref{THM:LWZ20}, noting that $4G$ is a $(12k+4)$-edge-connected graph.

		\medskip
		\noindent    
		{\ref{F0}.} For this claim, we aim to prove that there exists a sufficiently large $s=s(G)$ such that $(G, \sigma)$ admits a $({2ks-2}, {ks-s})$-flow. 
		Our claim then follows by observing that $\frac{2ks-2}{(k-1)s} <\frac{2k}{k-1}$. In order to get a $({2ks-2}, {ks-s})$-flow in $(G, \sigma)$, using Theorem~\ref{thm:orientation}, it would be sufficient to find a $(\mathbb{Z}_{4ks-4}, \beta_{_{ks-1}})$-orientation on $(2s-2)G$. One may easily check that $\beta_{_{ks-1}}$ is a parity-compliant $(4ks-4)$-boundary of $(2s-2)G$.
		
		To this end we first build a graph $H$ by adding a vertex $z$ to the graph $(2s-2)G$ and connecting it to each vertex of $G$ with $6k-8$ parallel edges. Observe that $d_{\scriptsize H}(z)=(6k-8)|V(G)|$.
		
		Next we extend $\beta_{_{ks-1}}$ to $z$ by defining $\beta_{_{ks-1}}(z)= 0$, but with slight abuse of notation we use the same name $\beta_{_{ks-1}}$. 
		By the construction of $H$ (from $(2s-2)G)$), the degree of each vertex in $H$ is even. It is then easily verified that the extended $\beta_{_{ks-1}}$ is a parity-compliant $(4ks-4)$-boundary of $H$.
		
		Next we shall apply Theorem~\ref{evenpartialextending-old} to obtain a $(\mathbb{Z}_{4ks-4}, \beta_{_{ks-1}} )$-orientation on $H$. To that end, we first consider the partial orientation $D_{z}$ at the vertex $z$ defined as follows: For each vertex $v$ of $G$, orient half of the edges connected to $z$ toward $v$ and the other half away from $v$. If we choose $s$ large enough, then we have $d(z) \leq (4ks-4)-2+|\beta_{_{ks-1}}(z)|$. Here the choice of $s$ depends on the order of $G$. For each subset $A$ of $V(G)$ with $|V(G)\setminus A|>1$, since $(2s-2)G$ is $(6ks-6k)$-edge-connected, we have at least $6ks-6k$ edges connecting $A$ to $V(G)\setminus A$. Note that, since $z\not \in A$, there are $(6k-8)|A|$ edges connecting $z$ to $A$. Thus $d_{\scriptsize H}(A)\geq 6ks-6k+(6k-8)|A| \geq 6ks-8$, the inequality being the consequence of the fact that $|A|\geq 1$ and $k\geq 2$. Therefore, noting that $|\beta_{_{ks-1}}(A)|\leq 2ks-2$, we have that $d_H(A)\geq 6ks-8 \geq (4ks-4)-2+|\beta_{_{ks-1}}(A)|$. 
		
		As the conditions of Theorem~\ref{evenpartialextending-old} are satisfied for $H$ with $z$ being the special vertex, we have an extension of $D_{z}$ to a $(\mathbb{Z}_{4ks-4}, \beta_{_{ks-1}})$-orientation $D$ on $H$. We claim that the restriction of $D$ to $(2s-2)G$ is also a $(\mathbb{Z}_{4ks-4}, \beta_{_{ks-1}})$-orientation on it. This is the case because, for each vertex $v$, the number of edges oriented to $z$ from $v$ and oriented to $v$ from $z$ are chosen to be the same.  We can then apply Theorem~\ref{thm:orientation} to get a $({2ks-2}, {ks-s})$-flow in $(G, \sigma)$. 
	\end{proof}
	
	One of the key points of the proof in the previous theorem is to consider $2G$ or $4G$ so that the $\beta$ function we consider is a parity-compliant boundary. If the graph itself had no odd-degree vertex, then we can directly work with $G$. In the case that $G$ is $(6k-2)$-edge-connected, this leads to a slight improvement on the bound for the flow index as follows.
	
	\begin{theorem}\label{thm:Eulerian}
		For any signed Eulerian graph $(G, \sigma)$, if $G$ is $(6k-2)$-edge-connected, then $\Phi_c(G, \sigma)\leq \frac{4k}{2k-1}$.
	\end{theorem}
	
	\begin{proof}
		Applying Theorem~\ref{THM:LWZ20} to $G$ we get a $(\mathbb{Z}_{4k}, \beta_{k})$-orientation on $G$, where $\beta_{k}\equiv 2k\cdot d^+_{\scriptscriptstyle \hat{G}}(v)~\pmod {4k}$. The claim then is concluded by the equivalence of part (1) and part (4) in Lemma~\ref{lem:EulerianEquivalence}.
	\end{proof}

	\section{Application to planar graphs}\label{sec:Application to planar graphs}
	
	As mentioned in the introduction, the circular flow index of a signed plane graph is equal to the circular chromatic number of its dual. Thus we have the following corollary of Theorem~\ref{thm:mainFlow}. 
	
	\begin{corollary}\label{coro:PlanarHighGirth}
		Given a signed planar graph $(G, \sigma)$ and an integer $k\geq2$, the following claims hold.
		\begin{enumerate}[label=(\arabic*)]
			\item\label{g-1} If $G$ is of girth at least $3k-1$, then $\chi_c(G, \sigma) \leq \dfrac{2k}{k-1}$.
			\item\label{g0} If $G$ is of girth at least $3k$, then $\chi_c(G, \sigma) < \dfrac{2k}{k-1}$.
			\item\label{g+1} If $G$ is of girth at least $3k+1$, then $\chi_c(G, \sigma) \leq \dfrac{4k+2}{2k-1}$.
		\end{enumerate}	
	\end{corollary}
	
	For the dual of Theorem~\ref{thm:Eulerian}, we will present a stronger result replacing the girth condition with the negative girth condition. To this end we first present two lemmas.

	\begin{lemma}\label{lem:D_z0Extendable}
		Given a positive integer $k$, a graph $G$ and a vertex $z$ of it, assume that the cut $(\{z\}, V(G)\setminus \{z\})$ is of size at most $6k-2$, but every other cut $(X, X^c)$ is of size at least $6k-2$. Then given any parity-compliant $4k$-boundary $\beta$ of $G$ and any orientation $D_z$ of the edges incident to $z$ satisfying that $\overleftarrow{d_{ \scriptscriptstyle D_{\scriptscriptstyle z}}}(z)-\overrightarrow{d_{\scriptscriptstyle D_{\scriptscriptstyle z}}}(z) \equiv \beta(z) ~\pmod {4k}$, $D_{z}$ can be extended to a $(\mathbb{Z}_{4k}, \beta)$-orientation on $G$.
	\end{lemma}

	\begin{proof}		
		Assume that $\beta$ and $D_z$ are given as in the lemma. Given an orientation $D$, let $-D$ be the orientation obtained from $D$ by flipping every arc. Then $D$ is a $(\mathbb{Z}_{4k}, \beta)$-orientation on $G$ if and only if $-D$ is a $(\mathbb{Z}_{4k}, -\beta)$-orientation on $G$. Hence, we may assume that $\beta(z)\in \{0, 1, \ldots, 2k\}$.
		
		Our goal is to apply Theorem~\ref{evenpartialextending-old}. Observing that since $\beta(A)$ is assumed to be in $\{0, \pm 1, \ldots, \pm 2k\}$, we have $|\beta(A)|\leq 2k$ for every $A\subset  V(G)$. Thus the condition $d(A)\ge 4k-2 + |\beta(A)|$ holds for every choice of $A$ except $A=\{z\}$, for which the conditions are not required. We only need to consider the condition on the vertex $z$. If $d(z) \leq 4k -2 +\beta(z)$, then we can directly apply Theorem~\ref{evenpartialextending-old}. So we assume $d(z) -4k +2 -\beta(z)> 0$. Combining the fact that $d(z) \leq 6k-2$ and $\beta(z)>0$, we have $4k-2< d(z)\leq 6k-2$. We aim to modify both $D_z$ and $\beta$ following the operation defined in Definition~\ref{def:D_e} so that by Observation~\ref{obs:D,B-->D_e,B_e} we can apply Theorem~\ref{evenpartialextending-old} to the new partial orientation and the new boundary function. What remains to do is to modify $D_z$ and $\beta$ to $D^{*}_z$ and $\beta^{*}$ respectively such that $\beta^*(z)=d(z)-4k+2$ and $D^*$ achieves $\beta^*$ at $z$.

		Since $\overleftarrow{d_{ \scriptscriptstyle D_{\scriptscriptstyle z}}}(z)-\overrightarrow{d_{\scriptscriptstyle D_{\scriptscriptstyle z}}}(z) \equiv \beta(z) ~\pmod {4k}$, $\overleftarrow{d_{ \scriptscriptstyle D_{\scriptscriptstyle z}}}(z)+\overrightarrow{d_{\scriptscriptstyle D_{\scriptscriptstyle z}}}(z)=d(z) \leq 6k-2$ and $0\leq \beta (z)\leq 2k$,  we have the following three possibilities: $\overleftarrow{d_{ \scriptscriptstyle D_{\scriptscriptstyle z}}}(z)-\overrightarrow{d_{\scriptscriptstyle D_{\scriptscriptstyle z}}}(z)\in \{\beta(z), \beta(z)+4k, \beta(z)-4k\}$.

		\begin{enumerate}[label=(\arabic*)]
			\item $\overleftarrow{d_{ \scriptscriptstyle D_{\scriptscriptstyle z}}}(z)-\overrightarrow{d_{\scriptscriptstyle D_{\scriptscriptstyle z}}}(z)= \beta(z)$.
			
			In this case, we need to flip $\frac{d(z)-\beta(z)}{2}-(2k-1)$ many in-arcs at $z$ in $D_z$ to out-arcs. A first comment here is that since $d(z)+\beta(z)=2\overleftarrow{d_{ \scriptscriptstyle D_{\scriptscriptstyle z}}}(z)$, $d(z)+\beta(z)$ and hence, $d(z)-\beta(z)$, is an even number, thus $\frac{d(z)-\beta(z)}{2}$ is an integer. As we have assumed $d(z) -4k +2 -\beta(z)> 0$, we know that $\frac{d(z)-\beta(z)}{2}-(2k-1)$ is a positive integer. We also need to show that there are indeed $\frac{d(z)-\beta(z)}{2}-(2k-1)$ in-arcs to flip. That is to claim that $\frac{d(z)-\beta(z)}{2}-(2k-1)\leq \overrightarrow{d_{\scriptscriptstyle D_{\scriptscriptstyle z}}}(z)$. Recalling that $\overrightarrow{d_{\scriptscriptstyle D_{\scriptscriptstyle z}}}(z)=\frac{d(z)-\beta(z)}{2}$, this is trivially the case. 
			
			Let $D^{*}_z$ be the partial orientation obtained from $D_z$ by flipping $\frac{d(z)-\beta(z)}{2}-(2k-1)$ in-arcs of $z$ and let $\beta^{*}$ be the corresponding boundary function obtained as defined in the Observation~\ref{obs:D,B-->D_e,B_e}. We have $\beta^{*}(z)=\beta(z)+2(\frac{d(z)-\beta(z)}{2}-(2k-1))=d(z) -4k+2$, noting that this value must be in the set $\{0, \pm 1, \ldots, \pm k\}$. Thus we may apply Theorem~\ref{evenpartialextending-old} to extend the partial orientation $D^{*}_{z}$ to a $(\mathbb{Z}_{4k}, \beta^{*})$-orientation $D^{*}$. Then flipping back the arcs we had flipped, by Observation~\ref{obs:D,B-->D_e,B_e}, we get to the orientation $D$ as an extension of $D_{z}$.

			\item $\overleftarrow{d_{ \scriptscriptstyle D_{\scriptscriptstyle z}}}(z)-\overrightarrow{d_{\scriptscriptstyle D_{\scriptscriptstyle z}}}(z)= \beta(z)+4k$.

			In this case, we have $\overleftarrow{d_{ \scriptscriptstyle D_{\scriptscriptstyle z}}}(z)=\frac{d(z)+\beta(z)}{2}+2k$.
			Since $d(z) -4k+2-\beta(z) >0$ and that we have assumed $\beta(z)\geq 0$, we have $d(z) \geq 4k-1$. Thus $\overleftarrow{d_{ \scriptscriptstyle D_{\scriptscriptstyle z}}}(z)\geq 4k$. We aim at flipping $(4k-1)-\frac{d(z)-\beta(z)}{2}$ in-arcs at $z$, to do so we need to show that this number is not a negative one and that it does not exceed the total number of in-arcs at $z$. The former is the case because $4k-1\leq d(z) \leq 6k-2$ and $0\leq \beta(z) \leq 2k$. The latter is the case because  $4k-1-\frac{d(z)-\beta(z)}{2}< 2k+\frac{d(z)+\beta(z)}{2}=\overleftarrow{d_{ \scriptscriptstyle D_{\scriptscriptstyle z}}}(z)$. 
			
			Let $D^{*}_z$ be the partial orientation from $D_{z}$ by flipping $4k-1-\frac{d(z)-\beta(z)}{2}$ in-arcs at $z$, and let $\beta^{*}$ be the parity-compliant $2k$-boundary obtained from $\beta$ following Observation~\ref{obs:D,B-->D_e,B_e}. Noting that $\beta^{*}(z)$ must be in the set $\{0, \pm 1, \ldots, \pm 2k\}$, and consider the limits $4k-1\leq d(z) \leq 6k-2$, we can compute the value of $\beta^{*}(z)=d(z) -4k+2$. Theorem~\ref{evenpartialextending-old} can then be applied to extend $D^{*}_z$ to an orientation $D^{*}$ which is a $(\mathbb{Z}_{4k}, \beta^*)$-orientation on $G$. Then following observation \ref{obs:D,B-->D_e,B_e} we get the required orientation $D$ on $G$.

			\item $\overleftarrow{d_{ \scriptscriptstyle D_{\scriptscriptstyle z}}}(z)-\overrightarrow{d_{\scriptscriptstyle D_{\scriptscriptstyle z}}}(z)= \beta(z)-4k$.
			
			In this case, we have  $\overrightarrow{d_{\scriptscriptstyle D_{\scriptscriptstyle z}}}(z)=\frac{d(z)-\beta(z)}{2}+2k$. Recall that $d(z) -\beta(z) \geq 2k$, thus we may flip a set of $\frac{d(z)-\beta(z)}{2} +1$ in-arcs at $z$. After so many flips, and following Observation~\ref{obs:D,B-->D_e,B_e}, we have $\beta^{*}(z)\equiv d(z)+2 ~\pmod{4k}$. Since $4k-2 <d(z) \leq 6k-2$, and $\beta^{*}(z) \in \{0, \pm 1, \ldots, \pm 2k\}$, we must have $\beta^{*}(z)=d(z)-4k+2$. Thus, as before, we may apply Theorem~\ref{evenpartialextending-old} on $D^{*}_z$ and $\beta^*$ to get the orientation $D^{*}$ from which, using Observation~\ref{obs:D,B-->D_e,B_e}, we get the required orientation $D$. 
		\end{enumerate}
		This completes the proof.
	\end{proof}

	The other lemma we need is the bipartite analogue of the folding lemma from \cite{NRS13}.
	
	\begin{lemma}{\em [Bipartite folding lemma]} \label{lem:Folding Lemma}
		Let $\hat{G}$ be a signed bipartite plane graph and let $2k$ be the length of its shortest negative cycle. Assume that $C$ is a facial cycle that is not a negative $2k$-cycle. Then there are vertices $v_{i-1}, v_i,v_{i+1}$, consecutive in the cyclic order of the boundary of $C$, such that identifying $v_{i-1}$ and $v_{i+1}$, after a possible switching at one of the two vertices, the resulting signed graph remains a signed bipartite plane graph whose shortest negative cycle is still of length $2k$.
	\end{lemma}

	By applying this lemma repeatedly, one gets a homomorphic image of $\hat{G}$ which is also a signed bipartite plane graph in which every facial cycle is a negative cycle of length exactly $2k$. Based on this fact and Lemma~\ref{lem:D_z0Extendable}, we are ready to prove the following. 
	
	\begin{theorem}\label{thm:NegativeBipartite6k-2}
		Every signed bipartite planar graph of negative-girth at least $6k-2$ admits a circular $\frac{4k}{2k-1}$-coloring. 
	\end{theorem}
	
	\begin{proof}
		Assume to the contrary that $(G, \sigma)$ is a minimum counterexample with respected to $|E(G)|+|V(G)|$. By Lemma~\ref{lem:Folding Lemma}, we may assume that $(G, \sigma)$ is a signed bipartite plane graph of negative-girth $6k-2$ in which each facial cycle is a negative $(6k-2)$-cycle and $(G, \sigma)$ admits no circular $\frac{4k}{2k-1}$-coloring. Let $\hat{G}^*=(G^*, \sigma^*)$ be the dual signed plane graph of $(G, \sigma)$. Hence, the signed graph $\hat{G}^*$ is Eulerian, $(6k-2)$-regular and moreover, each of its negative cut has size at least $6k-2$. If $G^*$ is $(6k-2)$-edge-connected, then we are done by Theorem~\ref{thm:Eulerian}. Thus we may assume that $\hat{G}^*$ has a positive even cut of size strictly less than $6k-2$. Let $(X, X^c)$ be such a cut with $X$ being inclusion-wise minimal among all the possibilities. That is to say, for every proper subset $Y$ of $X$ we have $|(Y, Y^c)|\geq 6k-2$. 
		
		Let $\hat{H}$ denote the signed subgraph of $\hat{G}^*$ induced by $X$. Observing that $|X|\geq 2$ and $\hat{H}$ is connected, we consider $\hat{G}^*/\hat{H}$ where all the edges of $\hat{H}$ are contracted but the remaining edges get their signs from $\sigma^*$. We claim that $\hat{G}^*/\hat{H}$ admits a circular $\frac{4k}{2k-1}$-flow. Otherwise, its dual signed graph, which is a proper subgraph of $(G, \sigma)$ (because $\hat{H}$ is connected), admits no circular $\frac{4k}{2k-1}$-coloring, contradicting to the minimality of $(G, \sigma)$. By the equivalence of $(1)$ and $(4)$ in  Lemma~\ref{lem:EulerianEquivalence}, we know that $\hat{G}^*/\hat{H}$ admits a $(\mathbb{Z}_{4k}, \beta')$-orientation with $\beta'(v)\equiv 2k\cdot d^+(v)~\pmod{4k}$ for $v\in V(\hat{G}^*/\hat{H})$. Let $D'$ be such a $(\mathbb{Z}_{4k}, \beta)$-orientation on $\hat{G}^*/\hat{H}$.
		
		Next we build a signed graph $\hat{G}_1$ from $\hat{G}^*$ by identifying all vertices in ${X}^c$ to a vertex $z$, deleting resulting loops, but keeping all parallel edges. Note that $d_{G_1}(z)=|(X, X^c)|<6k-2$ but for any other vertex subset $S\subset V(G_1)$, $|(S, S^c)|\geq 6k-2$. Let $D'_{z}$ be the orientation on the edges incident to $z$ (i.e., $E(X, X^c)$ in $\hat{G}^*$) induced by $D'$ and let $\beta''$ be a parity-compliant $4k$-boundary of $\hat{G}_1$ satisfying that $\beta''(v)\equiv 2k\cdot d^+(v)~\pmod{4k}$ for $v\in V(\hat{G}_1)$. Note that $$\beta''(z)\equiv 2k\cdot d^+(z) \equiv  \overrightarrow{d_{ \scriptscriptstyle D'_{\scriptscriptstyle z}}}(z)-\overleftarrow{d_{\scriptscriptstyle D'_{\scriptscriptstyle z}}}(z) \equiv \overleftarrow{d_{ \scriptscriptstyle D'_{\scriptscriptstyle z}}}(z)-\overrightarrow{d_{\scriptscriptstyle D'_{\scriptscriptstyle z}}}(z) ~\pmod{4k}$$ and thus we know that $D'_{z}$ achieves $\beta''$ at $z$. We may now apply Lemma~\ref{lem:D_z0Extendable} to extend the partial orientation $D'_{z}$ and obtain a $(\mathbb{Z}_{4k},\beta')$-orientation $D''$ on $\hat{G}_1$. 
		
		Combining $D'$ and $D''$, and also combining $\beta'$ (restricted to $X^c$) and $\beta''$ (restricted to $X$), we get a parity-compliant $4k$-boundary $\beta$ of $\hat{G}^*$ with $\beta(v)\equiv 2k\cdot d^+(v)~\pmod{4k}$ and the $(\mathbb{Z}_{4k}, \beta)$-orientation $D$ on $\hat{G}^*$. Using Lemma~\ref{lem:EulerianEquivalence} once again, we conclude that $\hat{G}^*$ admits a circular $\frac{4k}{2k-1}$-flow. Equivalently, as its dual, $(G, \sigma)$ must admit a circular $\frac{4k}{2k-1}$-coloring, contradicting the fact that this was a (minimum) counterexample to our claim.
	\end{proof}
	
	It is shown in \cite{NW22+} that the signed cycle $C_{\!\scriptscriptstyle -2k}$ is the bipartite circulant $(4k,2k-1)$-clique, that is to say, a signed bipartite graph $(G, \sigma)$ satisfies $\chi_c(G, \sigma)\leq \frac{4k}{2k-1}$ if and only if $(G, \sigma)$ admits a homomorphism to $C_{\!\scriptscriptstyle -2k}$. Thus we have the following corollary.
	
	\begin{corollary}
		Every signed bipartite planar graph of negative-girth at least $6k-2$ admits a homomorphism to $C_{\!\scriptscriptstyle -2k}$.
	\end{corollary}

	\section{Conclusion and Questions}
	Recall a restatement of Tutte's $5$-flow conjecture is that every $2$-edge-connected signed graph admits a circular $10$-flow. 
	It has been proved in \cite{PZ03} that for any rational number $r$ between $2$ and $5$, there exists a graph $G$ with circular flow index being $r$. Thus by Lemma~\ref{lem:FlowOfT_2(G)} and considering the signed graph $T_2(G)$, we have the following.
	
	\begin{proposition}
		For any rational number $r\in [2, 10]$, there exists a $2$-edge-connected signed graph whose circular flow index is $r$.
	\end{proposition}
	
	For $3$-edge-connected signed graphs, on the one hand, we have a $6$-flow theorem (Theorem~\ref{thm:3-edgeconnected6-flow}), and on the other hand, we do not know any example whose circular flow index is larger than $5$.
	
	\begin{problem}\label{prob:5-flow}
		Is it true that every $3$-edge-connected signed graph admits a circular $5$-flow?
	\end{problem}
	
	The technique of \cite{S81} can be adapted to reduce the above problem to $3$-edge-connected signed cubic graphs.
	
	\medskip
	
	Tutte's $4$-flow conjecture asserts that every $2$-edge-connected Petersen-minor-free graph admits a circular $4$-flow. Similar to the restatement of Tutte's $5$-flow conjecture, by applying Proposition \ref{prop:2phi} and Lemma \ref{lem:FlowOfT_2(G)}, this conjecture can be reformulated as follows.

	\begin{conjecture}{\em [Tutte's $4$-flow conjecture restated]}\label{conj:8-flow}
		Every $2$-edge-connected signed Petersen-minor-free graph admits a circular $8$-flow.
	\end{conjecture}
	
	If the conjecture is true, then a potential strengthening would be to use the notion of minors of signed graphs, and consider $2$-edge-connected signed graphs with no $(P, -)$-minor where $P$ is the Petersen graph. 
	
	Note that the Petersen graph has circular flow index $5$. Furthermore, an infinite family of snarks with circular flow index $5$ are built in \cite{MR06}. For each such snark $G$, by Lemma~\ref{lem:FlowOfT_2(G)}, $T_2(G)$ is a $2$-edge-connected signed graph whose circular flow index is $10$. Considering the same $T_2$-operation, and in combination with the result of \cite{EMT16}, it follows that to decide if an input $2$-edge-connected signed graph has circular flow index strictly smaller than $10$ is an NP-hard problem.  
	
	\medskip
	
	In Theorem~\ref{thm:3-edgeconnected6-flow}, we have seen that every $4$-edge-connected signed graph admits a circular 4-flow. This upper bound is tight because a sequence of signed bipartite plane simple graphs of girth $4$ with circular chromatic numbers approaching $4$ is given in \cite{NWZ21}. Their duals provide a sequence of $4$-edge-connected signed Eulerian plane graphs with circular flow indices approaching $4$. The following however remains open:
	
	\begin{problem}\label{prob:Phi_c=4}
		Is there a $4$-edge-connected signed graph $(G,\sigma)$ satisfying that $\Phi_c(G, \sigma)=4$? 
	\end{problem}
	
	A result of \cite{KNNW21+} can be restated as: a signed graph verifying Problem~\ref{prob:Phi_c=4} positively cannot be planar and Eulerian. Moreover, Theorem~\ref{thm:3spanningtree<4flow} implies that such an example cannot be 6-edge-connected either.
	
	\medskip
	
	Tutte's $3$-flow conjecture states that every $4$-edge-connected graph admits a nowhere-zero $3$-flow. Later, M. Kochol proved in \cite{K01} that Tutte’s $3$-flow conjecture is equivalent to the seemingly weaker statement that every $5$-edge-connected graph admits a nowhere-zero $3$-flow.
	Motivated by Kochol's result, we may propose the next conjecture which implies Tutte's $3$-flow conjecture if it is true.
	
	\begin{conjecture}\label{conj:3-flow}
		Every $5$-edge-connected signed graph admits a circular $3$-flow.
	\end{conjecture}

	In conclusion, generalizing the notion of circular flow index of graphs to signed graphs, we have provided upper bounds for the circular flow indices of several classes of signed graphs, in particular for the class of $k$-edge-connected signed graphs. Our bounds are not shown to be tight and in fact, in most cases, they are not even expected to be tight. However, the problem of finding tight bounds captures some of the most challenging open problems in graph theory.
	
	Another venue for strengthening our results would be to prove same or similar bounds for larger classes of signed graphs. For example, in the case of graphs, using the splitting lemma~\cite{Zh02}, in some results the condition of edge-connectivity can be replaced with the condition of odd-edge-connectivity. In case of signed graphs perhaps the value of $c_{\scriptscriptstyle 00}(G, \sigma)$ (i.e., the size of a smallest positive even cut) does not have a strong affect on the upper bounds. Thus as a refinement of our study, given a triple $(a_{_{01}}, a_{_{10}}, a_{_{11}})$, one may ask for the best possible upper bound on the circular flow indices of signed graphs $(G, \sigma)$ satisfying $c_{\scriptscriptstyle 01}(G,\sigma)\geq a_{_{01}}$, $c_{\scriptscriptstyle 10}(G,\sigma)\geq a_{_{10}}$, and $c_{\scriptscriptstyle 11}(G,\sigma)\geq a_{_{11}}$. For the triple $(\infty, \ell, \infty)$, the question is about circular flow indices of graphs with odd-edge-connectivity at least $\ell$. In another special case of $(\ell, \infty, \infty)$, the question is about the circular flow indices of signed Eulerian graphs all whose negative cuts are of size at least $\ell$. In relation to orientations and homomorphisms to negative cycles, upper bounds of special interests are numbers in the form of $\frac{2k}{k-1}$. Thus given a positive integer $k>1$, one may ask: under which edge-connectivity conditions a signed graph is sure to admit a circular $\frac{2k}{k-1}$-flow? The best known results are summarized in Table~\ref{table:phi(k)}.
	
	\begin{center}
		
		\captionof{table}{Edge Connectivity and Circular Flow Index in signed Graphs}
		\label{table:phi(k)}
		\renewcommand{\arraystretch}{1.3}
		\begin{tabular}{ | l | l | l | l |  }
			\hline
			Edge-Connectivity & Conjectures  & Known bounds \\ \hline
			2 & $\Phi_c\le 10$ (Conj.~\ref{conj:10-flow}) & $\Phi_c\le 12$ (Thm.~\ref{thm:12-flow}) \\ \hline
			3 & $\ast$
			& $\Phi_c\le 6$ (Thm.~\ref{thm:3-edgeconnected6-flow}) \\ \hline
			4 & $\ast$ & $\Phi_c\le 4$ (Thm.~\ref{thm:3-edgeconnected6-flow})   \\ \hline
			5 & $\Phi_c\le 3$ (Conj.~\ref{conj:3-flow}) & --- \\ \hline
			6 & & $\Phi_c<4$ (Thm.~\ref{thm:3spanningtree<4flow}) \\ \hline
			$\cdots$ & $\cdots$  & $\cdots$ \\ \hline	
			$3k-1$ & $\ast$  & $\Phi_c\le \frac{2k}{k-1}$ (Thm.~\ref{thm:mainFlow}) \\ \hline
			$3k$ & $\ast$  & $\Phi_c< \frac{2k}{k-1}$ (Thm.~\ref{thm:mainFlow}) \\ \hline
			$3k+1$ & $\ast$  & $\Phi_c\le \frac{4k+2}{2k-1}$ (Thm.~\ref{thm:mainFlow}) \\ \hline
			$6k-2$+Eulerian & $\ast$  & $\Phi_c\le \frac{4k}{2k-1}$ (Thm.~\ref{thm:Eulerian}) \\ \hline
		\end{tabular}
	\end{center}

	\medskip
	{\bf Acknowledgment.} This work is partially supported by the following grants: National Natural Science Foundation of China (Nos. 12222108, 12131013) and the Young Elite Scientists Sponsorship Program by Tianjin (No. TJSQNTJ-2020-09); ANR (France) project HOSIGRA (ANR-17-CE40-0022); European Union's Horizon 2020 research and innovation program under the Marie Sklodowska-Curie grant agreement No 754362, and NSERC-Canada grant number R611450; National Natural Science Foundation of China grant NSFC 11971438 and U20A2068 and by Zhejiang Natural Science Foundation grant ZJNSF LD19A010001.

\end{document}